\documentclass[journal,twoside,web]{ieeecolor}
\usepackage{generic}
\usepackage[utf8]{inputenc}
\usepackage{graphicx}
\usepackage{amsmath}
\usepackage{longtable,tabularx}
\usepackage{graphicx}      
\usepackage{amssymb}
\usepackage{amsmath}
\usepackage{wrapfig}
\usepackage{pgfplots}
\pgfplotsset{compat=1.18}
\usepgfplotslibrary{fillbetween,patchplots}
\usepgfplotslibrary{groupplots}
\usepgfplotslibrary{colormaps}
\usetikzlibrary{decorations.pathreplacing,arrows.meta,positioning}
\usepackage{amsmath} 
\usepackage{amsfonts}
\let\theoremstyle\relax

\newtheorem{theorem}{Theorem}
\newtheorem{assumption}{Assumption}
\newtheorem{proposition}{Proposition}[section]
\newtheorem{definition}{Definition}[section]

\newtheorem{remark}{Remark}

\newtheorem{corollary}{Corollary}

\usepackage{float} 
\usepackage{comment}

\usepackage{tikz}
\usetikzlibrary{shapes,arrows,calc,positioning}
\tikzstyle{bigblock} = [draw, fill=blue!20, rectangle, 
    minimum height=6em, minimum width=8em]
\tikzstyle{medblock} = [draw, fill=blue!20, rectangle, 
    minimum height=4em, minimum width=4em]    
\tikzstyle{mux} = [draw, fill=black!20, rectangle, 
    minimum height=5em, minimum width=0.1em]    
\tikzstyle{smallblock} = [draw, fill=blue!20, rectangle, 
    minimum height=2em, minimum width=3em]
    
\tikzstyle{data_block} = [draw, fill=green!20, rectangle, 
    minimum height=2em, minimum width=3em]
\tikzstyle{ops_block} = [draw, fill=blue!20, rectangle, 
    minimum height=2em, minimum width=3em]    
\tikzstyle{est_block} = [draw, fill=red!20, rectangle, 
    minimum height=2em, minimum width=3em]    
    
\tikzstyle{sum} = [draw, fill=blue!20, circle, node distance=1cm,minimum height=0.5cm]
\tikzstyle{signal} = [coordinate]
\tikzstyle{pinstyle} = [pin edge={to-,thin,black}]
\tikzstyle{block} = [draw, fill=blue!20, rectangle, 
    minimum height=3em, minimum width=9em]
\tikzstyle{blockS} = [draw, fill=blue!20, rectangle, 
    minimum height=3em, minimum width=4em]    
\tikzstyle{input} = [coordinate]
\tikzstyle{output} = [coordinate]
\usetikzlibrary{matrix}

\usetikzlibrary{positioning}
\usetikzlibrary{math}

\usepackage{hyperref}
\hypersetup{
    colorlinks,
    linkcolor={blue!100!black},
    citecolor={blue!50!black},
    urlcolor={blue!80!black}
}

\newcommand{\bc}{\begin{center}}
\newcommand{\ec}{\end{center}}
\newcommand{\benum}{\begin{enumerate}}
\newcommand{\eenum}{\end{enumerate}}
\newcommand{\nn}{\nonumber}
\newcommand{\matl}{\left[ \begin{array}}
\newcommand{\matr}{\end{array} \right]}

\renewcommand{\matl}{\begin{bmatrix}}
\renewcommand{\matr}{\end{bmatrix}}

\newcommand{\matls}{\left[ \begin{smallmatrix}}
\newcommand{\matrs}{\end{smallmatrix} \right]}
\newcommand{\isdef}{\stackrel{\triangle}{=}}

\newcommand{\inv}{^{-1}}

\newcommand{\half}{\tfrac{1}{2}}

\newcommand{\rmI}{{\rm I}}

\newcommand{\rmT}{{\rm T}}

\newcommand{\rmd}{{\rm d}}

\newcommand{\BBR}{{\mathbb R}}

\newcommand{\SB}{{\mathcal B}}
\newcommand{\SC}{{\mathcal C}}

\newcommand{\SF}{{\mathcal F}}
\newcommand{\SG}{{\mathcal G}}

\newcommand{\SSS}{{\mathcal S}}

\newcommand{\SV}{{\mathcal V}}

\let\labelindent\relax
\usepackage{enumitem,amssymb}
\newlist{todolist}{itemize}{2}
\setlist[todolist]{label=$\square$}
\usepackage{pifont}
\usepackage{amsmath}
\usepackage{float}
\usepackage{stfloats}
\usepackage{amssymb}
\usepackage{tabularx}
\usepackage[table,xcdraw]{xcolor}
\usepackage{array}
\DeclareMathOperator{\atanh}{atanh}
\newcommand{\jp}[1]{\textcolor{blue}{#1}}
\renewcommand{\jp}[1]{\textcolor{black}{#1}}
\usepackage{hyperref}
\newcommand{\rotation}[2]{ {\substack{#1 \\ {\boldsymbol{\longrightarrow}} \\ #2}} }

\usepackage{booktabs}

\begin{document}



\title{A Constraint-Lifting Framework for \\ Safe and Stable Nonlinear Control}

\author{Jhon Manuel Portella Delgado
\thanks{Jhon Manuel Portella is a Postdoctoral Research Fellow with the Department of Aerospace Engineering, University of Michigan, Ann Arbor, MI 48109, USA.
This work was completed while he was a Ph.D. candidate in the Department of Mechanical Engineering, University of Maryland, Baltimore County, MD 21250, USA.
{\tt jportella@umbc.edu}}
and Ankit Goel
\thanks{Ankit Goel is an Assistant Professor in the Department of Mechanical Engineering, University of Maryland, Baltimore County, MD 21250.
{\tt ankgoel@umbc.edu}}}                                      
\maketitle
\begin{abstract}                          
    This paper presents a constraint-lifting control framework for designing stabilizing controllers that guarantee the forward invariance of a prescribed safe set.
    State-of-the-art safety-enforcing methods, such as control barrier functions (CBFs) and model predictive control (MPC), typically rely on solving constrained optimization problems in real time and therefore may not yield an explicit control law that guarantees constraint satisfaction under all conditions.
    In contrast, the proposed approach develops an explicit control law for a class of nonlinear systems that ensures both asymptotic stabilization of a desired equilibrium and safety preservation of a user-defined set.
    The central idea is to lift the constrained state space into an unbounded domain using a sigmoid-based diffeomorphic mapping, synthesize the controller in the transformed coordinates, and then map it back to the original coordinates.
    To address numerical conditioning near constraint boundaries, a special class of Lyapunov candidate functions, called \textit{sigmoid integral functions}, is introduced.
    A rigorous stability analysis, based on the Barbashi-Krasovskii-LaSalle invariance principle, establishes asymptotic convergence and safety guarantees.
    The efficacy of the proposed controller is demonstrated through a safe attitude-control problem.

\end{abstract}


    




\section{Introduction}
\jp{
The problem of enforcing state and output constraints in nonlinear systems has received significant attention in control theory due to its importance in safety-critical applications in aerospace and robotics.
}
Classical control theory primarily addresses the stabilization and command-following problems. 
State-of-the-art approaches for handling state constraints often use techniques such as saturation functions \cite{graichen2010handling} and reference governors \cite{garone2017reference, wang2023balanced, gilbert1995discrete, gilbert1999fast}.
\jp{
However, designing a controller that enforces state constraints while guaranteeing stability remains a challenging problem. %
While effective methods exist, they often involve practical trade-offs, such as the use of auxiliary constraint-enforcement layers (e.g., CBF-based filters), reliance on online optimization (e.g., MPC), or sensitivity near constraint boundaries (e.g., in BLF-based designs).
}
\jp{
A fundamental challenge arises when the constrained state has a relative degree greater than one with respect to the control input. In this case, existing approaches typically enforce safety through auxiliary constructions that impose additional constraints on higher-order states, yielding a forward invariant set that is a subset of the original safe set.
}
\jp{
This observation motivates the following question: can one guarantee forward invariance of the original safe set for systems subject to higher relative-degree constraints, without introducing auxiliary constraints?
}
This paper thus addresses the \textit{state-constrained control problem} by developing a control framework that ensures satisfaction of user-specified state constraints while maintaining closed-loop stability.
\jp{
In particular, this paper shows that such non-conservative invariance can be achieved through a constraint-lifting transformation combined with a Lyapunov-based design.
}

\begin{figure}[t]
    \centering
    \begin{tikzpicture}
\begin{scope}[shift={(0.5,-1)}] 
\begin{axis}[
    view={45}{30},          
    domain=-3:3,
    y domain=-4:4,
    samples=40,
    samples y=40,
    axis lines=none,
    colormap={cartoongray}{
        rgb=(0.95,0.95,0.95)
        rgb=(0.80,0.80,0.80)
        rgb=(0.60,0.60,0.60)
        rgb=(0.40,0.40,0.40)
    },
    shader=interp,
    opacity=0.9,
    z buffer=sort
]

\addplot3[
    surf,
    draw=black!40,
    line width=3pt
] {0.5*x^2 + ln(cosh(y/0.5))};

\draw[->, thick, red] (axis cs:0,0,0) -- (axis cs:-3.5,1.8,0) node[below=3pt] {$z_1$};
\draw[->, thick, red] (axis cs:0,0,0) -- (axis cs:0,3.5,0) node[below=3pt] {$z_2$};

\addplot3[
    very thick,
    smooth,
    green!70!black,
    mark=none
] coordinates {
    (0,0,10.5)
    (2,1,6.8)
    (-1,-1,6.8)
    (1,0.5,2.5)
    (0,0,0)
};

\addplot3[only marks, mark=*, mark options={fill=green!70!black}] coordinates {(0,0,10.5)};
\addplot3[only marks, mark=*, mark options={fill=blue}] coordinates {(0,0,0)};

\end{axis}
\end{scope}

\node at (6,3.8) {\parbox{2.2cm}{\centering\textbf{Unconstrained\\dynamics}}};

\draw[fill=pink!20,draw = black, thick] (-1,-1) -- (0,-2) -- (1,-1) -- (0,0) -- cycle;


\node[above,yshift=1pt] at (0.4,0) {\textbf{Safe Set}};

\node[below, yshift= -50pt ,xshift=30pt] at (0,0) {\parbox{2.2cm}{\centering\textbf{Constrained\\dynamics}}};

\draw[green!70!black, very thick, smooth] plot coordinates {
    (-0.5,-1.2)
    (-0.3,-0.4)
    (-0.1,-1.85)
    (0.3,-1.4)
    (0,-1)
};

\fill[green!70!black] (-0.5,-1.2) circle (2pt);
\fill[blue] (0,-1) circle (2pt);

\draw[->, thick, red] (0,-1) -- (-1.2,0.2) node[below=3pt] {$x_1$}; 
\draw[->, thick, red] (0,-1) -- (1.2,0) node[below=3pt] {$x_2$};    

\draw[->, thick, blue, bend left=40, >=Stealth] 
    (0.8,-1.3) .. controls (5,-1.5) .. (4,0.3) 
    node[midway, below] {$\phi$};

\draw[->, thick, blue, bend left=30, >=Stealth] 
    (2,1) .. controls (-1.5,0.5) .. (-0.5,0) 
    node[midway, above] {$\psi$};

\end{tikzpicture}
    \caption{Enforcing forward invariance via the constraint lifting framework.}
    \label{fig:constraint_sketch}
\end{figure}
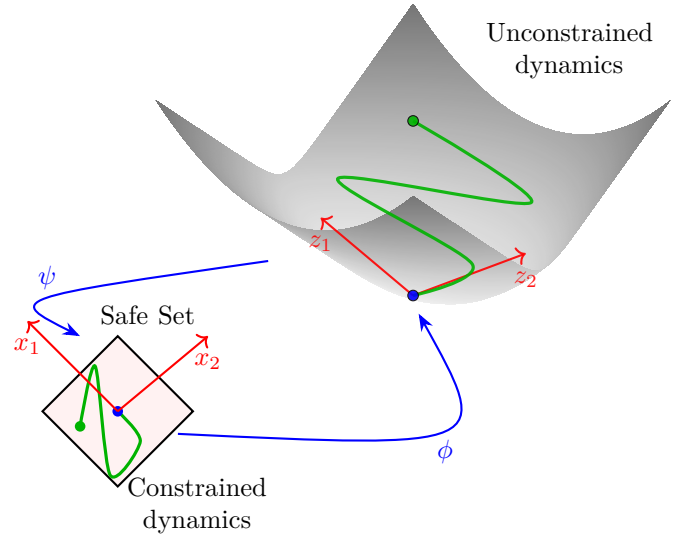

\jp{While optimal control frameworks can be formulated to handle constraints, constraint satisfaction in such approaches is enforced through the feasibility of an associated optimization problem, which also optimizes a performance criterion. }
\jp{Model predictive control (MPC), a subclass of optimal control techniques, provides a systematic framework that integrates constraint handling with performance optimization through an online optimization problem \cite{kim2019backstepping,kim2018integration}.}
However, MPC presents several practical challenges in implementation. %
First, solving the full nonlinear optimization problem can be computationally demanding, potentially limiting its applicability in real-time embedded implementations \cite{SCHLAGENHAUF20207026}. %
To address this, linearization techniques are often employed to reduce the optimization problem to a quadratic program \cite{wang2020}; however, this approach introduces modeling approximations. %
Furthermore, because constraints are enforced over a prediction horizon, modeling errors arising from these approximations can affect constraint satisfaction. %
Finally, because solving the optimization problem at each sampling instant is computationally expensive, maintaining performance guarantees in fast or highly dynamic applications can be challenging.

Incorporating constraints into Lyapunov-based formulations represents an alternative approach to handling constraints. 
Nevertheless, the presence of constraints usually introduces singularities in the resulting control synthesis.
Barrier Lyapunov functions (BLFs), a class of these methods, provide a framework to handle constraints by using Lyapunov candidate functions that grow unbounded near constraint boundaries \cite{li2023adaptive,tee2009barrier,soukkou2023tuning}. 
\jp{Although BLFs can theoretically prevent constraint violations, they typically introduce singular terms into the resulting controller that grow unbounded as the state approaches the constraint boundary, which can lead to numerical sensitivity and large control inputs\cite{ngo2005integrator,niu2024adaptive}.}
Mitigating these singularities requires careful gain tuning, and stability guarantees often rely on conservative assumptions \cite{Sadeghzadeh2023}.
%
%
Furthermore, as shown in \cite{yin2019barrier,li2020barrier,tee2011control}, the BLF framework enforces constraints on the error states rather than directly on the system states, thereby limiting its application. 

Recently, control barrier functions (CBFs) have emerged as an alternative tool for enforcing safety constraints
\cite{di2014stabilizing,zhu2024adaptive,taylor2022safe,Kim2023,Kim2025}.
CBFs \cite{romdlony2016stabilization,ames2019control,wu2019control} are based on Nagumo's invariance principle \cite{nagumo1942lage,blanchini1999set}, which introduced the notion of \textit{lower integrals} to guarantee positive forward invariance of a prescribed set.
In particular, Nagumo's invariance principle yields a \textit{safety inequality} that ensures the system's state remains within the prescribed safe set if the inequality is satisfied. 
Note that the CBFs themselves do not generate a control signal; instead, they modify a control signal to enforce safety.
In practice, the CBFs are implemented as discrete-time safety filters that, at each step, solve a constrained optimization problem to find a signal as close as possible to the nominal control signal while satisfying the safety inequality \cite{ames2016control}. 
Since the CBFs are necessarily implemented using sample data, this complicates the closed-loop stability analysis. 
Ad-hoc techniques such as fixed margins can be used to mitigate inter-sample violations, but they generally degrade performance. 
Control Lyapunov functions (CLFs), similar to CBFs, have been proposed to guarantee simultaneous safety and stability, but these approaches ultimately rely on solving a constrained optimization problem to compute the control signal. 
\jp{Furthermore, when the constrained variable has relative degree greater than one with respect to the control input, the construction of control barrier functions becomes nontrivial and typically requires higher-order barrier formulations \cite{ong2024rectified}. 
Such formulations enforce safety through multiple barrier conditions, resulting in an invariant set defined by the intersection of several constraints, which may introduce conservatism. 
This structural limitation motivates the development of alternative approaches that can enforce invariance of the original constraint set directly.}

\begin{table*}[t]
\centering
\rowcolors{2}{gray!15}{white}
\begin{tabularx}{\textwidth}{|
    >{\centering\arraybackslash}p{1.5cm}|
    >{\arraybackslash}p{4.5cm}|
    >{\arraybackslash}X|
    >{\arraybackslash}X|}
\hline
\rowcolor{cyan!70!black}
\textcolor{white}{\textbf{Method}} & \textcolor{white}{\textbf{Central Idea}} & \textcolor{white}{\textbf{Limitations}} & \textcolor{white}{\textbf{Contrast with Proposed Method} }
\\
\hline
\textbf{MPC} &
Embeds constraints in an online optimization framework. &
High computational cost; requires solving a nonlinear or quadratic program at each step; prone to inter-sample violations. &
\jp{Provides an explicit control law without online optimization; invariance guarantees are established in continuous time.} \\
\hline
\textbf{BLF} &
Uses Lyapunov functions that diverge near constraint boundaries. &
Contains singular terms causing numerical ill-conditioning; constraints are often applied to error states. &
\jp{Employs sigmoid-integral Lyapunov functions to mitigate numerical growth near constraint boundaries.} \\
\hline
\textbf{CBF} &
Uses Nagumo’s principle to ensure forward invariance. &
\jp{Requires solving QPs at every instant. 
Furthermore, for constraints with relative degree greater than one, higher-order barrier constructions are required, leading to multiple barrier conditions and invariant sets defined through their intersection, which may introduce conservatism.} &
\jp{Yields an analytic control law with continuous-time invariance and stability guarantees under stated assumptions. 
Handles constraints on states with relative degree greater than one directly within the control design, without auxiliary higher-order barrier constructions.} \\
\hline
\textbf{State Transformation} &
Maps constrained space to unconstrained via a nonlinear diffeomorphism. &
Mappings diverge near boundaries, causing numerical instability and compounded ill-conditioning. &
\jp{Uses sigmoid-based maps with Lyapunov-based compensation to mitigate numerical growth near constraint boundaries.} \\
\hline
\multicolumn{4}{|c|}{\textbf{Comparison Across Methods}} \\
\hline
\textbf{Invariant Set Property} &
\jp{Set whose forward invariance is guaranteed by the method.} &
\jp{For higher relative degree constraints, the invariant set is defined through auxiliary conditions and is generally a subset of the original constraint set.} &
\jp{Guarantees forward invariance of the original safe set without introducing auxiliary constraints or set reduction.} \\
\hline
\end{tabularx}
\caption{
Comparison of constraint-enforcing control techniques with the proposed method. 
\jp{The table highlights key differences in constraint enforcement and invariant set guarantees, particularly for states with relative degree greater than one.}
}
\label{tab:comparison_methods}
\end{table*}

An alternative approach, which also motivates the constraint-lifting framework developed in this paper, uses state transformations to embed constraints into the system dynamics \cite{guo2014backstepping, huang2020output, zhu2024adaptive}. 
In particular, continuous functions that diverge to infinity at the constraint boundaries are used to transform the system dynamics. 
However, these nonlinear mappings may induce numerical instability near the constraint boundaries.
Furthermore, constraints that are applied to the entire state result in the composition of several such nonlinear functions, leading to even more pronounced numerical instabilities. 
Ad hoc techniques, such as switching controllers, have been investigated to address numerical instability \cite{sampei2003nonlinear}.

This paper introduces the constraint-lifting framework for designing a control law that simultaneously enforces constraints and guarantees closed-loop stability for a class of nonlinear systems. 
\jp{Similar to state-transformation techniques, the proposed framework converts a state-constrained control problem into an equivalent unconstrained one, as shown in Figure \ref{fig:constraint_sketch}, by \textit{lifting} the constrained state space into an unconstrained state space, thereby converting a control problem defined over a constrained domain into an equivalent problem in which the transformed state evolves without explicit constraints.}
In contrast to existing approaches, however, this work introduces a novel Lyapunov function that mitigates the numerical instabilities commonly observed in prior methods.
Specifically, a controller is synthesized for the unconstrained system using the classical backstepping approach with the proposed Lyapunov function, called the \textit{sigmoid integral function}.
This function exhibits quadratic behavior near the origin and linear growth away from it, a key property that eliminates the exponentially growing terms responsible for numerical ill-conditioning near constraint boundaries in conventional Lyapunov-based control designs.
Unlike MPC and CBF-based approaches, the proposed method yields an explicit control law rather than requiring the solution of an optimization problem at each sampling instant, thereby eliminating inter-sample constraint violations.
Furthermore, unlike BLF- and state-transformation-based techniques, the proposed method avoids the use of singular terms, thereby preventing numerical ill-conditioning.
The comparison of existing constraint-enforcing control techniques with the proposed approach is summarized in Table \ref{tab:comparison_methods}.
\jp{Note that the results are established under nominal model assumptions; robustness to uncertainty is not addressed in this work.}

The main contributions of this paper are as follows. 
First, a constraint-lifting, sigmoid-based diffeomorphic map is introduced to reformulate the constrained control problem as an equivalent unconstrained one.
Second, a Lyapunov function, called the sigmoid integral function, is proposed, which exhibits quadratic behavior near the origin and linear behavior away from it, thereby eliminating exponentially growing terms introduced by the state transformation.
Third, a rigorous stability analysis of the resulting closed-loop system is provided.
Finally, the effectiveness of the proposed method is demonstrated through the constrained attitude control problem.
\jp{
The development in this paper focuses on a two-level system structure for clarity of presentation.
Nevertheless, the proposed framework extends to systems with higher relative degree, in which the control synthesis follows a recursive structure analogous to backstepping, thereby increasing design complexity.
}

The paper is organized as follows. 
Section \ref{sec:problem formulation and preliminaries} describes the notation, formulates the control problem, states the necessary assumptions, and introduces the \textit{constraint-lifting function} and the \textit{sigmoid integral functions} used to construct the safety-enforcing controller.
Section \ref{sec:controller} presents the main result and the stability analysis of the closed-loop system.
Section \ref{sec:justification_sigmoid_integral} presents the motivation for introducing the sigmoid integral functions.
Section \ref{sec:example} demonstrates the application of the proposed method in the constrained attitude control problem.
Finally, the paper concludes with a discussion in Section \ref{sec:conclusions}.

\section{Problem Formulation and Preliminaries}
\label{sec:problem formulation and preliminaries}
This section establishes the notation, formulates the control problem, and introduces the sigmoid-function-based mappings used in the controller synthesis.

\subsection{Notation}
\label{sec:notation}
Let $q \in \mathbb{R}^n,$ define
\begin{align}
    D(q)
        &\isdef
            \text{diag}
            \left(
                q_1,
                q_2,
                \ldots,
                q_n
            \right)
            \in 
            \mathbb{R}^{n \times n}.
\end{align}
If, for $i=1,\ldots, n,$ $q_i \neq 0,$ then define
    \begin{align}
        D_\rmI(q) 
            &\isdef 
                {\rm diag}(q_1\inv, \ldots , q_n\inv) \in \BBR^{n \times n}.
        \label{eq:Diagonal_inverse}
    \end{align}
Note that $D_\rmI(q)\inv = D(q).$
%
The set of all $q \in \BBR^n$ such that $|q_i| < 1$ for all $i$ is denoted by $\SC$, equivalently,
\begin{align}
    \SC 
        \isdef
            \{x \in \BBR^n \colon \| x\|_\infty < 1 \}.
\end{align}
Note that $\SC$ is a unit norm ball induced by the infinity norm.
The set of $n$-dimensional real vectors with strictly positive components is denoted by $\BBR^n_{+}$, that is, 
\begin{align}
    \BBR^n_{+}
        \isdef
            \{ x \in \mathbb{R}^n : x_i > 0,\ \forall i = 1, \dots, n \}.
\end{align}

\subsection{System Description}
\label{sec:problem_formulation}
Consider the system 
\begin{align}
    \dot x_1 
        &= 
            g_1(x_1) x_2,
    \label{eq:x1_dot}
    \\
    \dot x_2 
        &= 
            f_2(x_1, x_2) + g_2(x_1, x_2) u,
    \label{eq:x2_dot}
\end{align}
where $x_1, x_2 \in \mathbb{R}^n$ denote the states, 
and $u \in \mathbb{R}^n$ is the control input.
Let $g_1:\mathbb{R}^n \to \mathbb{R}^{n \times n},$ $f_2:\mathbb{R}^n \times \mathbb{R}^n \to \mathbb{R}^n,$ and $g_2:\mathbb{R}^n \times \mathbb{R}^n \to \mathbb{R}^{n \times n}$ be continuously differentiable functions of sufficient order.

The objective is to design a control law that asymptotically \jp{stabilizes} a desired state $x_{1\rmd}$ and ensures that the system states remain bounded within user-specified limits.
Specifically, the states $(x_1, x_2)$ are required to remain inside a safe set $\SSS$ defined as
\begin{align}
\SSS
    &\isdef
        \{
        (x_1, x_2) \in \BBR^n \times \BBR^n \colon
        |x_1| < \overline{x}_1,
        |x_2| < \overline{x}_2
        \},
    \label{eq:safe_set}
\end{align}
where $\overline{x}_1, \overline{x}_2 \in \BBR^n_{+}$ denote the user-specified upper bounds on $x_1$ and $x_2$, respectively.
In other words, the objective is to design a control law that renders $(x_{1d},0)$ is asymptotically stable and guarantees the forward invariance of $\SSS.$
\jp{
In addition, the objective is to preserve forward invariance of the original safe set $\SSS$ directly, without introducing auxiliary constraints on higher-order states.
}

The following assumptions are required for controller synthesis. 
\begin{assumption}
    \label{ass:function_zero}
    The function $f_2(x_1,x_2) = 0$ if and only if $x_2 = 0.$
\end{assumption}

\begin{assumption}
    \label{ass:function_invert}
    The matrices $g_1(x_1)$ and $g_2(x_1,x_2)$ are nonsingular for $(x_1,x_2) \in \SSS.$ 
\end{assumption}

\jp{
\begin{remark}
Assumption \ref{ass:function_zero} is used to ensure that the set $\{\dot V = 0\}$ corresponds only to the desired equilibrium in the Lyapunov analysis. In particular, it guarantees that the condition $f_2(x_1,x_2)=0$ implies $x_2=0$, which is required to conclude asymptotic convergence via the invariance principle.
If this condition is relaxed, the set $\{\dot V = 0\}$ may contain additional trajectories, and the invariance analysis requires a more detailed characterization of the resulting invariant set. 
This does not fundamentally alter the control design but introduces additional technical steps to establish convergence properties.
\end{remark}
}

\jp{
\begin{remark}
Assumption \ref{ass:function_invert} ensures that the input mappings remain nonsingular within the safe set $\SSS$, which is required for the well-posedness of the control law and the invertibility of the associated transformations in the control design. 
This condition guarantees that the control input has authority over all state directions within $\SSS$, and is closely related to the notion of local controllability or full actuation in nonlinear control systems. 
Such assumptions are standard in inversion-based methods, including backstepping and feedback linearization. 
In practice, this implies that the safe set is selected to exclude configurations where control effectiveness is lost.
\end{remark}
}

\subsection{Conservatism in Higher-Order Barrier Methods}

\jp{Control barrier functions (CBFs) provide a systematic framework for enforcing forward invariance of a prescribed safe set based on Nagumo's condition. 
For constraints with relative degree one, the resulting safety condition depends directly on the control input, allowing the invariance condition to be enforced through a single inequality constraint.}

\jp{
However, when the constrained state has a relative degree greater than one with respect to the control input, the standard CBF construction is no longer directly applicable, since the control input does not appear in the first derivative of the constraint function.
}

\jp{
To address this limitation, higher-order control barrier function (HOCBF) formulations introduce additional constraints by recursively differentiating the barrier function until the control input explicitly appears. 
As a result, a sequence of barrier conditions must be satisfied simultaneously to ensure forward invariance.
Consequently, the forward invariant set is defined by the intersection of multiple constraint sets, each corresponding to a higher-order derivative condition.
This intersection imposes additional restrictions on higher-order states, such as velocities, even when the original constraint is defined only in terms of position.
Therefore, the resulting invariant set may be a strict subset of the original constraint set.
This effect introduces a form of conservatism, since admissible trajectories that satisfy the original constraint may be excluded by the additional derivative-based conditions.
}

\jp{
In contrast, the objective of this work is to enforce forward invariance of the original constraint set directly, without introducing auxiliary constraints on higher-order states.
}

\subsubsection{Illustrative Double-Integrator Example}
\label{sec:double_integrator_example}
\jp{To illustrate the conservatism introduced by higher-order barrier constructions, consider the double integrator}
\begin{align}
    \dot x_1 = x_2, \quad
    \dot x_2 = u,
    \label{eq:doubleIntegrator}
\end{align}
\jp{with the state constraint $h(x) \isdef 1 - x_1^2 \ge 0.$}
\jp{The corresponding safe set is}
\begin{align}
    \SC_0 \isdef \{x \in \mathbb R^2 : h(x) \ge 0\}
    =
    \{x \in \mathbb R^2 : |x_1| \le 1\}.
\end{align}

\jp{Since \(h\) has relative degree two with respect to the input \(u\), the control input does not appear in \(\dot h\), and a standard first-order CBF condition cannot be imposed directly. 
A higher-order barrier construction therefore introduces an additional constraint involving \(x_2\), which defines another set \(\SC_1\). The resulting forward invariant set is then given by
$\SC_0 \cap \SC_1,$
rather than by \(\SC_0\) alone. In particular, \(\SC_1\) imposes state-dependent restrictions on the velocity \(x_2\), even though the original safety constraint is posed only on \(x_1\). Therefore, the forward invariant set under the higher-order barrier construction is not simply \(h(x) \ge 0\), and may be a strict subset of the original safe set.
Figure \ref{fig:hocbf_double_integrator} shows the forward invariant set under a higher-order barrier construction, which shows that the admissible set is smaller than the original safe set $\SC_0$.
}

\begin{figure}[t]
    \centering
    \includegraphics[width=\columnwidth]{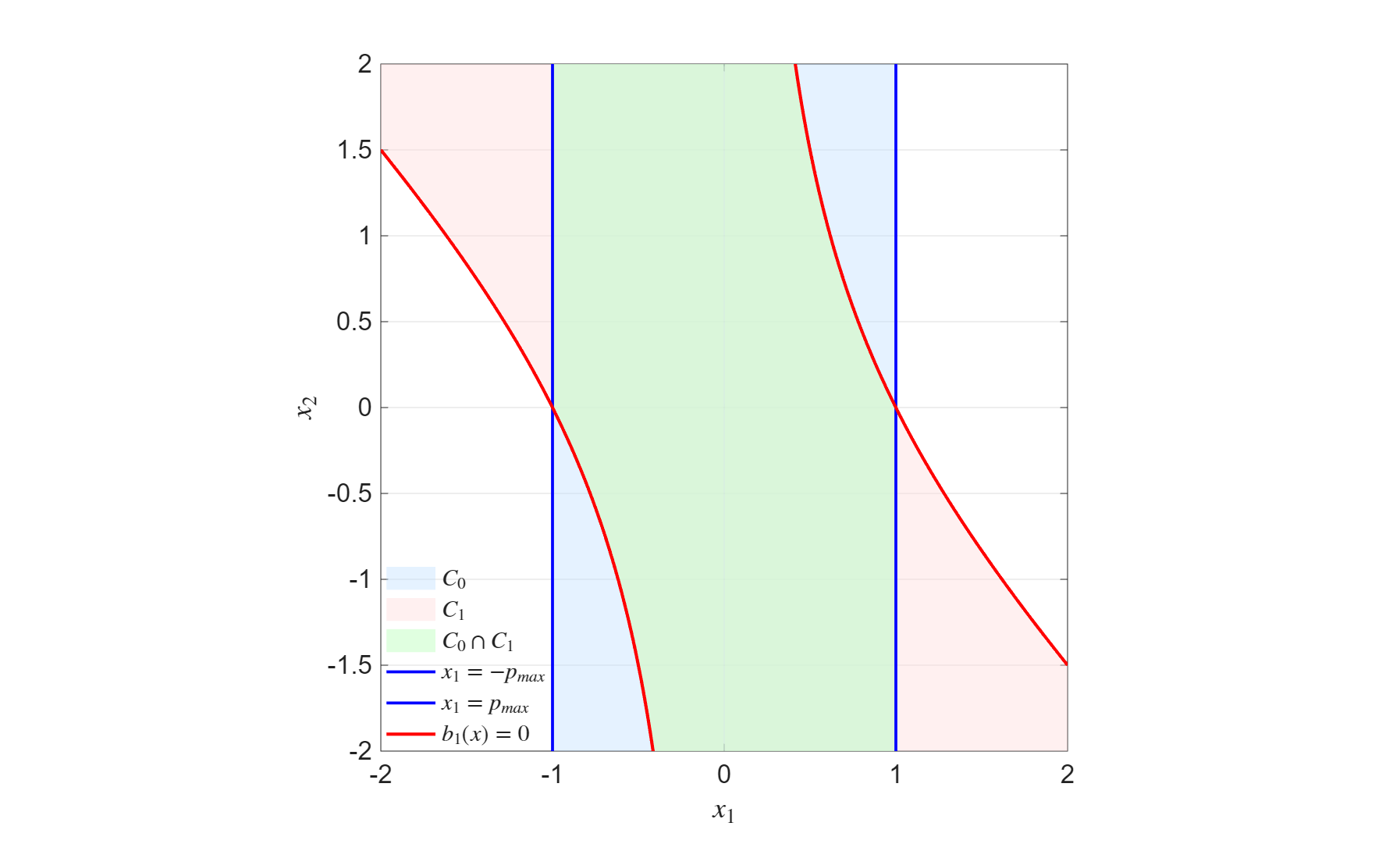}
    \caption{
    \jp{Forward invariant set for the double integrator under a higher-order barrier construction. 
    The admissible set is $\SC_0 \cap \SC_1$, which may be a strict subset of $\SC_0$.}
    }
    \label{fig:hocbf_double_integrator}
\end{figure}

\subsection{Sigmoid and Associated Functions}
\label{sec:sigmoid}
This section provides a brief review of the \textit{sigmoid functions} used to transform the state $x$ for the construction of a control law that ensures asymptotic stability of the desired state $x_{1\rmd}$ and guarantees forward invariance of $\SSS$.
In addition, these sigmoid functions are used to construct the \textit{constraint-lifting function} and to define a class of Lyapunov functions, called \textit{sigmoid integral functions}, that will later facilitate the controller synthesis.
\subsubsection{Sigmoid Functions}

We revisit the definition and key properties of \textit{sigmoid functions}, presented in \cite{menon1996characterization}. 
\jp{In particular, we will require the sigmoid function to be strictly increasing in order to ensure the existence of its inverse.}
\begin{definition}[\textbf{Simple sigmoids}]
    A function $\sigma \colon \BBR \to (-1,1)$ is a \textit{simple sigmoid} if it satisfies the following conditions.
    \begin{enumerate}
        \item $\sigma$ is a smooth function, that is, $\sigma(x)$ \jp{$\in$} $C^{\infty}.$
        \item $\sigma$ is an odd function, that is, $\sigma(-x) = -\sigma(x).$
        \item $\sigma$ has $y = \pm 1$ as horizontal asymptotes, that is, $\lim_{x \to \infty} \sigma(x) = 1.$
        \item $\sigma(x)/x$ is a completely convex function in $(0,1).$
    \end{enumerate}
\end{definition}

The following result, reproduced from \cite{menon1996characterization}, establishes that if the simple sigmoid is strictly increasing, then its inverse exists. 

\begin{proposition}
    Let $y = \sigma(x)$ be a strictly increasing simple sigmoid, that is, $\sigma'(x) > 0$ for all $x \in \mathbb{R}$, where $\sigma'$ denotes the derivative of $\sigma$. %
    Then,
    \begin{enumerate}
        \item $\eta \isdef \sigma^{-1}: (-1,1) \to \mathbb{R}$ exists.
        \item $\eta(y)$ is a strictly increasing function, analytic in the interval $(-1,1).$
        \item $\eta'(y) = 1/\sigma'(\eta(y)),$ where $\eta'$ and $\sigma'$ denote the first derivatives of $\eta$ and $\sigma,$ respectively.
        \item $\eta(y)/y$ is absolutely monotone in $(0,1).$
    \end{enumerate}
\end{proposition}

\subsubsection{Constraint-Lifting Function}
\label{sec:mapping function}
Using the strictly increasing sigmoid functions presented above, we define the \textit{constraint-lifting function}, which is used to convert a constrained control problem into an equivalent unconstrained problem.
In particular, the constraint-lifting function $\phi \colon \SC \to \BBR^n$ is constructed using the inverse of a strictly increasing sigmoid function, thereby providing a diffeomorphic transformation from the constrained state $x$ to the unconstrained state $z,$ which will be defined in Section \ref{sec:controller}.
Similarly, the mapping $\psi \colon \BBR^n \to \SC$, used to recover $x$ from $z$, is constructed using the strictly increasing sigmoid function.

In particular, the constraint-lifting function is 
\begin{align}
    \phi(\chi) \isdef 
        \matl
            \eta_1(\chi_1) &
            \eta_2(\chi_2) &
            \ldots &
            \eta_n(\chi_n)
        \matr^{\rmT} 
        \in 
        \BBR^n,
    \label{eq:phi_general}
\end{align}
where each $\eta_i$ is the inverse of a strictly increasing sigmoid function.
\jp{The functions $\eta_i$ may, in general, be chosen independently for each component.}
Note that the constraint-lifting function maps an element from $\SC$ to $\BBR^n.$
Similarly, the inverse of the constraint-lifting function is 
\begin{align}
    \psi(z) \isdef 
        \matl
            \sigma_1(z_1) &
            \sigma_2(z_2) &
            \ldots &
            \sigma_n(z_n)
        \matr^{\rmT} 
        \in 
        \BBR^n,
    \label{eq:psi_general}
\end{align}
where each $\sigma_i$ is the inverse of $\eta_i.$
Table \ref{tab:sat_functions} lists several strictly increasing sigmoid functions and their inverses in the second and the first column, respectively, and Figure \ref{fig:saturation_functions} shows these functions and their inverses.

\begin{table}[ht]
    \centering
    \resizebox{\columnwidth}{!}
    {
    \renewcommand{\arraystretch}{2} 
    \rowcolors{2}{gray!15}{white}
    \begin{tabularx}{1.07\columnwidth}{|>{\centering\arraybackslash}X|
                                      >{\centering\arraybackslash}X|
                                      >{\centering\arraybackslash}X|}
    \hline
    \rowcolor{cyan!70!black}
    \textcolor{white}{$\phi(x) = \eta(x)$} 
    & 
    \textcolor{white}{$\psi(z) = \sigma(z)$}
    &
    \textcolor{white}{$\SV(\zeta)$}
    \\
    \hline
        $ 
            \tan
            \left(
                \dfrac
                {
                    \pi
                }
                {
                    2
                }
                x
            \right)
        $
        & 
        $
            \dfrac
            {
                2
            }
            {
                \pi
            } 
            \text{atan}
            \left(
                z
            \right)
        $ 
        &
        $
            \zeta\,
            \text{atan}
            \left(
                \zeta
            \right)
            - \dfrac{1}{2}
            \log
            \left(
                1 
                + 
                \zeta^2
            \right)
        $
    \\
    \hline
        $
            \text{atanh}
            \left(
                x
            \right)
        $
        &
        $  
            \tanh
            \left(
                z
            \right)
        $
        &
        $
            \log
            \left(
                \cosh(\zeta)
            \right)
        $
    \\
    \hline
        $
            \dfrac{
                x
            }
            {
                1 - |x|
            }
        $
        &
        $
            \dfrac{
                z
            }
            {
                1 + |z|
            }
        $
        &
        $
                |\zeta|
                - \log
                \left(
                    1 + |\zeta|
                \right)
        $
    \\
    \hline
        $
            \dfrac
            {
                x
            }
            {
                \sqrt{1 - x^2}
            }
        $
        &
        $
            \dfrac
            {
                z
            }
            {
                \sqrt{1 + z^2}
            }
        $
        &
        $
                \sqrt{1 + \zeta^2}
                - 1
        $
    \\
    \hline
        $
            \dfrac{2}{\sqrt{\pi}}
            \text{erf}^{-1}
            \left(
                x
            \right)
        $
        &
        $
            \text{erf} 
            \left( 
                \dfrac{\sqrt{\pi}}{2} z
            \right)
        $
        &
        $
            \zeta\,
            \text{erf}
            \left(
                \dfrac{\sqrt{\pi}}{2}
                \zeta
            \right)
            +
            \dfrac{2}
            {
                \pi
            }
            \left(
                e^{-\dfrac{\pi}{4}\zeta^2}
                - 1
            \right)
        $
    \\
    \hline
        $
            \dfrac{2}{\pi}
            {\rm asinh}
            \left(
                \tan
                    \left(
                        \dfrac
                        {
                            \pi
                        }
                        {
                            2
                        }
                        x
                    \right)
            \right)
        $
        &
        $
            \dfrac{2}{\pi} \text{atan} 
            \left( 
                \sinh{ \left( \dfrac{\pi}{2} z \right) } 
            \right)
        $
        &
        $
                \dfrac{2}{\pi}
                \zeta
                \,
                \text{atan}
                \Big(
                    \sinh
                    {
                        \dfrac{\pi}{2}
                        \zeta
                    }
                    -
                    \displaystyle\int_0^{\dfrac{\pi}{2}\zeta}
                    r 
                    \text{sech}(r)
                    \rmd r
                \Big)
        $
    \\
    \hline
    \end{tabularx}
    }
    \caption{
    Various sigmoid functions, their invereses, and the corresponding sigmoid integral function.
    }
    \label{tab:sat_functions}
\end{table}

\begin{figure}[ht]
    \centering
    \includegraphics[width=1\columnwidth]{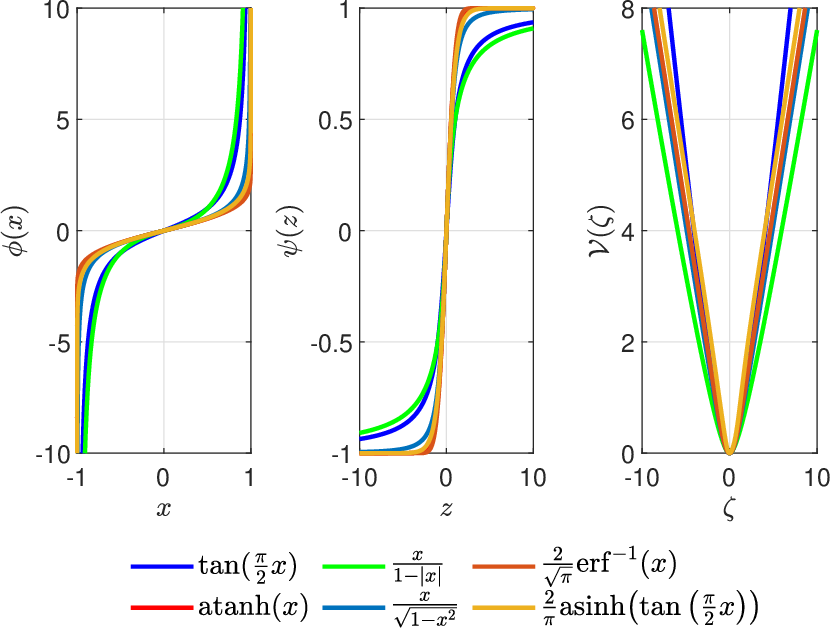}
    \caption{
        Various sigmoid functions, their invereses, and the corresponding sigmoid integral function.
        Note that the legend includes only the $\phi(x)$ functions, while the associated $\psi(z)$ and $\SV(\zeta)$ functions are shown using the same color scheme to maintain visual consistency.
    }
    \label{fig:saturation_functions}
\end{figure}

\begin{remark}
    For the controller synthesis in Section \ref{sec:controller} and the example in Section \ref{sec:example}, the same constraint-lifting function is used for both state variables to simplify the presentation.
    However, the components of the constraint-lifting function, that is, the sigmoid functions $\sigma_i$ in \eqref{eq:phi_general} and their corresponding inverses $\eta_i$ in \eqref{eq:psi_general} may be selected differently depending on the application.
\end{remark}

\begin{remark}
    Note that the constraint-lifting function is defined using the inverse of the sigmoid function, whereas its own inverse is constructed using the sigmoid function itself.
    This clarification is emphasized to avoid confusion arising from the use of the term ``inverse'' in both contexts.
\end{remark}

The following property of the constraint-lifting function is used in the stability analysis of the closed-loop system, presented in Section \ref{sec:stability_analysis}.
\begin{proposition}
    \label{rm:partial_invert_sat}
    Consider the constraint-lifting function \eqref{eq:phi_general}.
    Then, for all $x\in\SC,$ the Jacobian $\partial_x \phi(x)$ is nonsingular.
\end{proposition}
\begin{proof}
    The Jacobian $\partial_x \phi(x)$ is a diagonal matrix whose diagonal entries are strictly positive.  
    Since all diagonal elements are positive, $\partial_x \phi(x)$ is nonsingular.
\end{proof}

\subsubsection{Sigmoid Integral Function}
\label{sec:sig_integral}
Using the sigmoid functions presented above, we define the \textit{sigmoid integral function}.
The sigmoid integral function will be used to construct candidate Lyapunov functions in Section \ref{sec:controller}.
\begin{definition}
    \label{lineqr_quadratic_def}
    Let $\sigma:\mathbb{R}\to\mathbb{R}$ be a simple sigmoid function.
    The sigmoid integral function $\SV \colon \BBR \to [0, \infty)$ is defined as
    \begin{align}
        \SV(\zeta)
            \isdef
                \int_{0}^{{\zeta}} \sigma(s)   \rmd s.
        \label{eq:linear_quad_function_def}
    \end{align}
\end{definition}
The sigmoid integral function defined above has following easily verifiable properties. 
\begin{enumerate}
    \item $\SV(0)=0,$ and for all $\zeta \in \BBR \setminus \{ 0 \},$ $\SV(\zeta) > 0.$
    \item $\lim_{\zeta \to \pm \infty} \SV(\zeta) = \pm \infty.$
    \item $\partial_\zeta \SV\left(\zeta\right) = \sigma(\zeta)$, and thus the sigmoid integral function is continuous. 
\end{enumerate}
Note that the sigmoid integral functions $\SV(\zeta)$ exhibit quadratic behavior near $\zeta = 0$ and approach linear behavior for $|\zeta| \gg 0$.
Table \ref{tab:sat_functions} lists the sigmoid integral functions $\SV(\zeta)$ corresponding to each sigmoid function in the third column and Figure \ref{fig:saturation_functions} shows these sigmoid integral functions.

\section{Controller Synthesis}
\label{sec:controller}
This section develops a control law that guarantees the forward invariance of the safe set $\SSS$ for the system described by \eqref{eq:x1_dot}–\eqref{eq:x2_dot} while simultaneously guaranteeing asymptotic stability of the desired state $x_{1\rmd}$.

\subsection{State Transformation}
\label{sec:state_transformation}

The states subject to constraints are \textit{lifted} using the constraint-lifting function to obtain an unconstrained representation in the lifted state space.
The key idea of transforming the constrained dynamics into an equivalent unconstrained form is illustrated in Figure \ref{fig:constraint_lifting_transformation}.

\begin{figure}[ht]
    \centering

    \resizebox{0.9\columnwidth}{!}
    {
    \begin{tikzpicture}[scale=1,>=latex]

    \fill[pink!40,draw = black,thick] (-1,-1) rectangle (1,1);
    
    \draw[->,thick,red] (-1.5,0)--(1.5,0) node[anchor=west]{$x_1$};
    \draw[->,thick,red] (0,-1.5)--(0,1.5) node[anchor=south]{$x_2$};
    
    \draw[green!70!black, very thick, smooth]
      plot coordinates {
        (0.7,0.8)
        (-0.5,0.4)
        (-0.3,-0.8)
        (0.5,-0.5)
        (0.05,0.02)
        (0,0)
      };
    
    \filldraw[red] (0,0) circle (0.03);

    \fill[green!70!black] (0.7,0.8) circle (2pt);
    \fill[blue] (0,0) circle (2pt);

    \node[below, yshift=-30pt] at (0.8,0) {\textbf{safe set}};

    \draw[->, thick, black, bend left=40, >=Stealth] 
    (-0.5,0.4) .. controls (-1,2) .. (-1.5,1.5) 
    node[below, xshift=-15pt] {\parbox{2.2cm}{\centering\textbf{Constrained\\Dynamics}}};
    

    \begin{scope}[shift={(4.5,4.2)}]

        \draw[->,thick,red] (-2,0)--(2,0) node[anchor=west]{$z_1$};
        \draw[->,thick,red] (0,-2)--(0,2) node[anchor=south]{$z_2$};

        \draw[green!70!black, very thick, smooth]
          plot coordinates {
            (1.5,1.7)
            (-1,1)
            (-1.5,-1.5)
            (1,-1)
            (1,-0.5)
            (0,0)
          };

        \fill[green] (1.5,1.7) circle (2pt);
        \fill[blue] (0,0) circle (2pt);
        
        \draw[->, thick, black, bend left=10, >=Stealth] 
        (1,-0.5) .. controls (1.4,0) .. (1.5,-1.2) 
        node[below, xshift=0pt] {\parbox{2.2cm}{\centering\textbf{Unconstrained\\Dynamics}}};

        \draw[->,very thick, blue, bend right=10, >=Stealth] 
            (-3,-3.5) .. controls (-1.5,-3) .. (-0.5,-1.5) 
            node[midway, above, xshift=-30pt] {\parbox{2.2cm}{\centering\textbf{Constraint\\Lifting}}};
    \end{scope}


    \begin{scope}[shift={(0,4.2)}]

        \draw[-,ultra thick,black] (-1,-1.8)--(-1,1.8);
        \draw[-,ultra thick,black] (1,-1.8)--(1,1.8);
        \fill[pink!40] (-1,-1.8) rectangle (1,1.8);
        
        \draw[->,thick,red] (-1.5,0)--(1.5,0) node[anchor=west]{$x_1$};
        \draw[->,thick,red] (0,-2)--(0,2) node[anchor=south]{$z_2$};

        \draw[green!70!black, very thick, smooth]
          plot coordinates {
            (0.7,1.7)
            (-0.5,0.4)
            (-0.3,-1.7)
            (0.5,-0.5)
            (0,0)
          };
        
        \fill[green!70!black] (0.7,1.7) circle (2pt);
        \fill[blue] (0,0) circle (2pt);
    \end{scope}


    \begin{scope}[shift={(4.5,0)}]
        \draw[-,ultra thick,black] (-1.8,-1)--(1.8,-1);
        \draw[-,ultra thick,black] (-1.8,1)--(1.8,1);
        \fill[pink!40] (-1.8,-1) rectangle (1.8,1);

        \draw[->,thick,red] (-2,0)--(2,0) node[anchor=west]{$z_1$};
        \draw[->,thick,red] (0,-1.5)--(0,1.5) node[anchor=south]{$x_2$};

        \draw[green!70!black, very thick, smooth]
          plot coordinates {
            (1.7,0.8)
            (-1.7,0.4)
            (-0.3,-0.8)
            (0.7,-0.5)
            (0,0)
          };

        \fill[green!70!black] (1.7,0.8) circle (2pt);
        \fill[blue] (0,0) circle (2pt);
    \end{scope}
    
    \end{tikzpicture}
    }
    \caption{Illustration of the constraint-lifting concept, where the constrained dynamics in the original state space $(x_1,x_2)$ are transformed into equivalent unconstrained dynamics in the lifted state space $(z_1,z_2).$}
    \label{fig:constraint_lifting_transformation}
\end{figure}
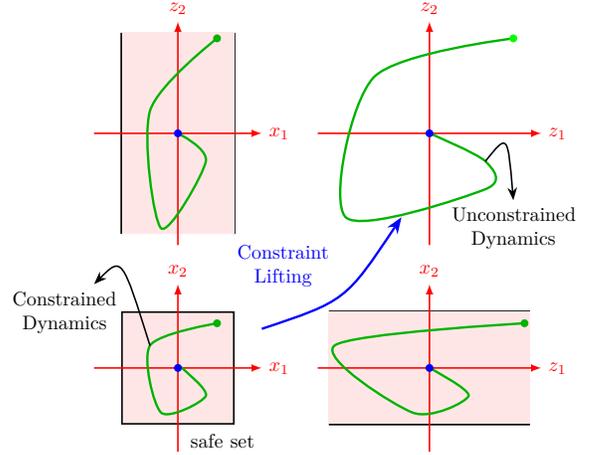

In particular, to facilitate controller synthesis, the states $(x_1, x_2)$ are first transformed to intermediate state variables $(\chi_1, \chi_2)$ such that the bounds on components of $x_1$ and $x_2$ are normalized to $\pm 1$.
This transformation ensures that each component of $\chi_1$ and $\chi_2$ lies within $(-1, +1)$ whenever $(x_1, x_2) \in \SSS$, that is, $\chi_1, \chi_2 \in \SC$.
Next, $(\chi_1, \chi_2)$ are transformed to $(z_1, z_2) \in \BBR^n$ through a constraint-lifting function, which are detailed in Section \ref{sec:mapping function}.
Due to the properties of the constraint-lifting function, while $\chi_1, \chi_2 \in \SC$, the transformed variables $z_1, z_2$ are unconstrained, that is, $z_1, z_2 \in \BBR^n$.
A stabilizing controller is then designed in the $z$-coordinates for the transformed $x_{1\rmd}$.
Finally, an additional transformation from $(z_1, z_2)$ to $(\zeta_1, \zeta_2)$ is introduced to simplify the notation used in mapping $z$ back to $x$ as well as the stability analysis.
The state transformations from $x$ to $z$ and the intermediate state variables $\chi$ and $\zeta$ are shown in Figure \ref{fig:state_transformation}.

\begin{figure}[ht]
\centering
    {%
    \begin{tikzpicture}[>={stealth'}, line width = 0.25mm]

    \node [input, name=ref]{};
    
    \node [smallblock, fill=green!20, rounded corners, right = 0.5cm of ref , 
           minimum height = 0.6cm, minimum width = 0.7cm] 
           (controller) {$D_\rmI(\overline x) x$};
    
    \node [smallblock, rounded corners, right = 2cm of controller, 
           minimum height = 0.6cm , minimum width = 0.7cm] 
           (integrator)            {$D(\overline x)\phi(\chi)$};
    
    \node [smallblock, fill=green!20, rounded corners, below = 1.cm of integrator, 
           minimum height = 0.6cm , minimum width = 0.7cm] 
           (system) {$D_\rmI(\overline x) z$};
           
    \node [smallblock, rounded corners, below = 1cm of controller, 
           minimum height = 0.6cm , minimum width = 0.7cm] 
           (integrator2)            {$D(\overline x)\psi(\zeta)$};

    \node [output, right = 1.0cm of system] (output) {};
    
    
    \node [input, left = 1.0cm of controller] (reference) {};
    
    \draw [->] (controller) -- node [above] {$\chi\in \SC$} (integrator);
    \draw [->] (integrator.0) -- +(1,0) node [above, xshift = -1em] {$z\in\BBR^n$} |-  (system.0);
    \draw [->] (system.180)  -- node [above] {$\zeta \in \BBR^n$} (integrator2.0);
    \draw [->] (integrator2.180) -- +(-1,0) |- node [above, xshift = 1em] {$x \in \SSS$} (controller.180);
    \end{tikzpicture}
    }  
    \caption{
        State transformations with the constraint-lifting function and its inverse used in the control synthesis.
    }
    \label{fig:state_transformation}
\end{figure}
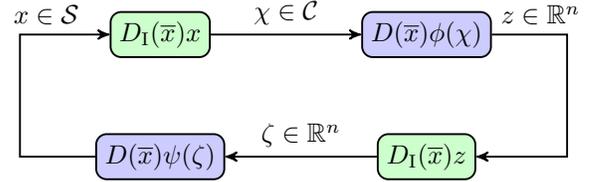

Define
\begin{align}
    \chi_1
        &\isdef
            D_{\rmI}(\overline{x}_1)
            x_1,
    \\
    \chi_2
        &\isdef
            D_{\rmI}(\overline{x}_2)
            x_2.
\end{align}
Note that the state variables $\chi_1$ and $\chi_2$ are defined to map the constraint boundaries to $\pm 1,$ that is, when the constraints are satisfied, then, $\chi_{1}, \chi_{2} \in \SC.$

Next, define
\begin{align}
    z_1 
        &\isdef
            D(\overline{x}_1)
            \phi(\chi_1),
    \label{eq:x1_to_z1}
    \\
    z_2 
        &\isdef
            D(\overline{x}_2)
            \phi(\chi_2),
    \label{eq:x2_to_z2}
\end{align}
where $\phi \colon \SC  \to \BBR^n $ is a constraint-lifting function as defined in Section \ref{sec:mapping function}.
\begin{remark}
    For simplicity, the same constraint-lifting function $\phi$ is used for both $z_1$ and $z_2$.
    This assumption is made without loss of generality, as the subsequent formulation can be easily extended to the case where distinct constraint-lifting functions are used for each state.
\end{remark}
Next, define
\begin{align}
    \zeta_{1}
        &\isdef
            D_{\rmI}(\overline{x}_1)
            z_1,
    \label{eq:zeta1_def}
    \\
    \zeta_{2}
        &\isdef
            D_{\rmI}(\overline{x}_2)
            z_2.
    \label{eq:zeta2_def}
\end{align}
The state variables $\zeta_1$ and $\zeta_2$ are defined to simplify the controller synthesis, as described below. 

Finally, define
\begin{align}
    x_1 
        &=
            D(\overline{x}_1)
            \psi(\zeta_1),
    \label{eq:z1_to_x1}
    \\
    x_2 
        &=
            D(\overline{x}_2)
            \psi(\zeta_2),
    \label{eq:z2_to_x2}
\end{align}
where $\psi:\mathbb{R}^n \to \SSS$ is an inverse of the constraint-lifting function as defined in Section \ref{sec:mapping function}.

With the state transformations defined above, the dynamics \eqref{eq:x1_dot}, \eqref{eq:x2_dot} in terms of the transformed state $z$ is
\begin{align}
    \dot{z}_1
        &=
            \SG_1(z_1,z_2), 
    \label{eq:z1_dot}
    \\
    \dot{z}_2
        &=
            \SF_2(z_1,z_2)  + \SG_2(z_1,z_2) u,
    \label{eq:z2_dot}
\end{align}
where 
\begin{align}
    \SG_1(z_1,z_2)
        &\isdef 
            \Phi(\zeta_1) 
            \psi(\zeta_2), 
    \label{eq:SG1def}
    \\
    \SF_2(z_1,z_2)
        &\isdef
            \left[ 
                \partial_{\chi_2} 
                \phi(\psi(\zeta_2))
            \right]
                f_2
                (
                    D(\overline{x}_1)
                    \psi(\zeta_1),
                    D(\overline{x}_2)
                    \psi(\zeta_2)
                ), 
    \label{eq:SF2def}
    \\
    \SG_2(z_1,z_2)
        &\isdef
            \left[ 
                \partial_{\chi_2} 
                \phi(\psi(\zeta_2))
            \right]
            g_2
                (
                    \jp{D(\overline{x}_1)}
                    \psi(\zeta_1),
                    D(\overline{x}_2)
                    \psi(\zeta_2)
                ),
    \label{eq:SG2def}
\end{align}
and
\begin{align}
    \Phi(\zeta_1) 
        \isdef
            \left[ \partial_{\chi_1} \phi(\psi(\zeta_1)) \right] 
            g_1(D(\overline{x}_1)\psi(\zeta_1))
            D(\overline{x}_2).
    \label{eq:Phi_def}.
\end{align}
Note that $\zeta_1$ and $\zeta_2$ are given by \eqref{eq:zeta1_def}, \eqref{eq:zeta2_def}.
Algebraic details of the derivation of dynamics in $z$ coordinates are shown in Appendix \ref{ap:derivation of z_dynamics}.

The following propositions establish key properties of the functions $\SG_1$, $\SF_2$, and $\SG_2$, which are subsequently used in the controller design and stability analysis.

\begin{proposition}
    \label{prop:equil_point_z1_dot}
    Consider $\SG_1(z_1,z_2) \in \BBR^{n}$ given by \eqref{eq:SG1def}.
    Then, $\SG_1(z_1,0) = 0.$
\end{proposition}
\begin{proof}
    Since $\psi(0)=0$ by definition, it follows from \eqref{eq:SG1def} that 
    $\SG_1(z_1,0) = \Phi(D_{\rmI}(\overline{x}_1) z_1) \psi(0)=0.$
\end{proof}

\begin{proposition}
    \label{prop:equil_point_z2_dot}
    Consider $\SF_2(z_1,z_2) \in \BBR^n$ given by \eqref{eq:SF2def}. 
    Then, $\SF_2(z_1,0) = 0.$
\end{proposition}
\begin{proof}
    Since $\psi(0)=0$ by definition, it follows from Assumption \ref{ass:function_zero} that $f_2
                (
                    D(\overline{x}_1)
                    \psi(\zeta_1),
                    0
                ) = 0,$
    which in turn implies $\SF_2(z_1,0) = 0.$
\end{proof}

\begin{proposition}
    \label{prop_SG2_invertible}
    Consider $\SG_2(z_1,z_2) \in \BBR^{n\times n}$ given by \eqref{eq:SG2def}.
    Then, $\SG_2(z_1,z_2)$ is nonsingular. 
\end{proposition}
\begin{proof}
    It follows from Proposition \ref{rm:partial_invert_sat} that $\partial_{\chi_2}\phi(\psi(\zeta_2)) = \partial_{\chi_2}\phi(\chi_2)$ is a diagonal matrix with strictly positive elements on the diagonal.
    Next, it follows from Assumption \ref{ass:function_invert} that $g_2
                (
                    \jp{D(\overline{x}_1)}
                    \psi(\zeta_1),
                    D(\overline{x}_2)
                    \psi(\zeta_2)
                ) $ is nonsingular.
    Since the product of two nonsingular matrices is a nonsingular matrix, it follows that $\SG_2(z_1,z_2)$ is nonsingular.
\end{proof}

\begin{proposition}
    \label{prop:PHI_invertible}
    Consider $\Phi(\zeta_1) \in \mathbb{R}^{n \times n}$ given by \eqref{eq:Phi_def}. 
    Then, $\Phi(\zeta_1)$ is nonsingular.
\end{proposition}
\begin{proof}
    It follows from Proposition \ref{rm:partial_invert_sat} that $\partial_{\chi_1}\phi(\psi(\zeta_1)) = \partial_{\chi_1}\phi(\chi_1)$ is a diagonal matrix with strictly positive elements on the diagonal. 
    Next, it follows from Assumption \ref{ass:function_invert} that
    $g_1
    (
        \jp{D(\overline{x}_1)}
        \psi(\zeta_1),
    ) $
    is nonsingular.
    Furthermore $D(\overline{x}_2)$ is positive definite.
    Since the product of nonsingular matrices is a nonsingular matrix, it follows that $\Phi(\zeta_1)$ is nonsingular.
\end{proof}

\subsection{Controller Synthesis}
\label{sec:ControlSynthesis}
Let $x_{1\rmd} \in \SSS$ be the desired value of $x_1.$
Let $\chi_{1\rmd} \in \SC$ and $z_{1\rmd}$ be the corresponding desired state values, that is, 
\begin{align}
    \chi_{1\rmd} 
        &\isdef 
        D_{\rmI}(\overline{x}_1)
        x_{1\rmd},
        \\
    z_{1\rmd} 
        &\isdef 
             D(\overline{x}_1)
            \phi(\chi_{1\rmd}).
    \label{eq:z1d}
\end{align}
Define the \jp{tracking error} 
\begin{align}
    e_1 
        &\isdef 
            z_1 - z_{1\rmd}.
    \label{eq:e1_def}
\end{align}
\subsubsection{\texorpdfstring{$e_1$}{e1} stabilization}
Consider the function
\begin{align}
    V_1
        &\isdef
            \dfrac{1}{2}
            e_1^{\rmT}e_1.
    \label{eq:V1def}
\end{align}
Differentiating \eqref{eq:V1def} and using \eqref{eq:z1_dot} yields
\begin{align}
    \dot{V}_1
        &=
            e_1^{\rmT}
            \SG_1(z_1,z_2).
    \label{eq:V1dot}
\end{align}

If $z_2$ is chosen such that $\SG_1(z_1,z_2) = - k e_1,$ where $k_1 > 0,$ then $\dot V_1 < 0.$
However, since $z_2$ is not the control input, $\SG_1(z_1,z_2)$ cannot be arbitrarily chosen. 
Instead, we define
\begin{align}
    e_2
        &\isdef
            \SG_1(z_1,z_2)
            -
            (- k_1e_1)
            =
                \Phi(\zeta_1) 
                \psi(\zeta_2)
                +
                k_1e_1.
    \label{eq:e2_def}
\end{align}
and design the control law to ensure that $e_2 $ converges to $ 0,$ thus implying that $\SG_1(z_1,z_2)$ converges to $ - k_1e_1.$ 
Substituting $\SG_1(z_1,z_2)$ from \eqref{eq:e2_def} into \eqref{eq:V1dot} yields
\begin{align}
    \dot{V}_1
        &=
            - k_1 e_1^{\rmT}e_1
            + e_2^{\rmT}e_1.
\end{align}

\subsubsection{\texorpdfstring{$e_2$}{e2} stabilization}

Next, consider the function
\begin{align}
    V
        &\isdef
            V_1
            +
            \sum_{i=1}^n
            \SV(\zeta_{2i}),
    \label{eq:LyapCandidate}
\end{align}
where $\SV$ is a sigmoid integral function given by \eqref{eq:linear_quad_function_def}.
As shown in Appendix \ref{ap:appendix_sigmoid integral definition}, note that
\begin{align}
    \rmd_t
    \sum_{i=1}^n
            \SV(\zeta_{2i})
        &= 
            \psi(\zeta_2)^{\rmT}
            D_{\rmI}(\overline{x}_2)
            \dot{z}_2. 
\end{align}
Thus,
\begin{align}
    \dot{V}
        &=
            - k_1e_1^{\rmT}e_1
            + e_2^{\rmT}e_1
            + 
            \psi(\zeta_2)^{\rmT}
            D_\rmI(\overline{x}_2)
        \nn \\
        & \quad
            \Big(
                    \SF_2(z_1,z_2) + \SG_2(z_1,z_2)u
            \Big).
\end{align} 
Substituting $\psi(\zeta_2)$ from \eqref{eq:e2_def} in the equation above yields
\begin{align}
    \dot{V}
        &=
            - k_1e_1^2
            + e_2^{\rmT}e_1
            + 
            (e_2^{\rmT} - k_1e_1^{\rmT})
            \Phi^{-\rmT}(\zeta_1)
        \nn \\
        & \quad
            D_\rmI(\overline{x}_2)
            \Big(
                    \SF_2(z_1,z_2) + \SG_2(z_1,z_2)u
            \Big).
\end{align}
Finally, choosing 
\begin{align}
    u
        &=
            - \SG_2(z_1,z_2)^{-1}
            \left(
                \SF_2(z_1,z_2)
                +
                k_2
                D(\overline{x}_2)
                \Phi^{\rmT}(\zeta_1)
                e_2
            \right)
    \label{eq:u}
\end{align}
yields
\begin{align}
    \dot{V}
        &=
            - k_1e_1^{\rmT}e_1
            + (1 + k_1k_2)e_2^{\rmT}e_1
            - k_2e_2^{\rmT}e_2 
        \nn \\
        &=
            - k_1e_1^{\rmT}e_1
            + (1 + k_1k_2)e_2^{\rmT}e_1
            - k_2e_2^{\rmT}e_2
        \nn \\
        & \quad
            + 2\sqrt{k_1k_2}e_1^{\rmT}e_2
            - 2\sqrt{k_1k_2}e_1^{\rmT}e_2
        \nn \\
        &=
            - \left(\sqrt{k_1}e_1 - \sqrt{k_2}e_2\right)^{\rmT}
            \left(\sqrt{k_1}e_1 - \sqrt{k_2}e_2\right)
        \nn \\
        & \quad
            + (1 + k_1k_2 - 2\sqrt{k_1k_2})e_2^{\rmT}e_1.
            \label{eq:V_dot}
\end{align}

Note that 
    If $k_2 = k_1^{-1},$ then,
    \begin{align}
        \dot{V}
                &=
                - \left(\sqrt{k_1}e_1 - \sqrt{k_2}e_2\right)^{\rmT}
                \left(\sqrt{k_1}e_1 - \sqrt{k_2}e_2\right).
                \label{eq:V_dot_k2condition}
    \end{align}
    Furthermore, if $k_1 e_1 = e_2,$ then $\dot V = 0.$

\subsection{Stability Analysis}
\label{sec:stability_analysis}
The stability analysis of the closed-loop dynamics with the proposed control law uses the Barbashin-Krasovskii-LaSalle invariance principle, which originally appeared as Theorems 1 and 2 in \cite{lasalle1960some}.
In particular, the Barbashin-Krasovskii-LaSalles's invariance principle is used to prove the asymptotic stability of the equilibrium point.

Propositions \ref{prop:equilibriumPoint}, \ref{prop:Omega_def}, \ref{prop:E_def}, and \ref{prop:M_def} introduced below define the sets used in the stability analysis and establish their properties.

\begin{proposition}
    \label{prop:equilibriumPoint}
    [\textbf{Equilibrium solution.}]
    Consider the system \eqref{eq:z1_dot}, \eqref{eq:z2_dot}.
    Let $z_{1\rmd} \in \BBR^n.$
    Then, the solution
    \begin{align}
        (z_1,z_2,u) 
            =
            \left(
                z_{1\rmd},
                0,
                0
            \right)
        \label{eq:equilibrium_point}
    \end{align}
    is an equilibrium point.
\end{proposition}
\begin{proof}
    If $z_2=0,$ then it follows from Proposition \ref{prop:equil_point_z1_dot} that $\SG_1(z_1,z_2)=0,$ which implies that $\dot z_1 = 0.$
    Furthermore, it follows from Proposition \ref{prop:equil_point_z2_dot} that $\SF_2(z_1,z_2)=0.$
    Moreover, if $u = 0,$ then $\dot z_2 = 0,$ implying that $(z_1, z_2, u) \equiv (z_{1\rmd},0,0)$ is an equilibrium point.
\end{proof}
\begin{proposition}
     \label{prop:Omega_def}
    Consider the system \eqref{eq:z1_dot}, \eqref{eq:z2_dot}.
    Consider the Lyapunov candidate function \eqref{eq:LyapCandidate}.
    Let $k_2 = k_1\inv.$
    Let $\ell > 0.$ 
    Define the sublevel set
    \begin{align}
        \Omega
            \isdef
                \{
                    (z_1,z_2,u) \in
                    \mathbb{R}^n
                    \times 
                    \mathbb{R}^n
                    \times 
                    \mathbb{R}^n
                    : V
                    \leq \ell
                \}.
        \label{eq:Omega_def}
    \end{align}
    Then, $\Omega$ is bounded and a positively invariant set.
\end{proposition}
\begin{proof}
    Since $V$ is radially unbounded and $\Omega$ is defined by $V \leq \ell,$ it follows from Theorem $2$ in \cite{lasalle1960some} that $\Omega$ is bounded. 
    Since $\dot{V} \leq 0$ and $V > 0,$ $V$ is non-increasing along the trajectory of the system. 
    Therefore, any trajectory starting in $\Omega$ cannot leave $\Omega,$ thus implying that $\Omega$ is a positively invariant set.
\end{proof}

\begin{proposition}
    \label{prop:E_def}
    Consider the system \eqref{eq:z1_dot}, \eqref{eq:z2_dot}.
    Consider the function \eqref{eq:LyapCandidate} and the positively invariant set $\Omega$, given in \eqref{eq:Omega_def}.
    Let $k_2 = k_1\inv.$
    Define $E \subset \Omega$ such that
    $z_2 = 0.$
    Then, $E$ is the set of all the points in $\Omega$ such that $\dot{V} = 0.$
\end{proposition}
\begin{proof}
    Note that it follows from \eqref{eq:zeta2_def} that $\zeta_2 = 0$ if and only if $z_2 = 0.$
    Next, $\zeta_2 = 0$ implies that $\psi(\zeta_2) = 0,$ which further implies that $\Phi(\zeta_1)  \psi(\zeta_2) = 0.$
    Next, it follows from \eqref{eq:e2_def} that $k_1e_1 = e_2.$
    Finally, since $k_2 = k_1\inv,$ it follows from \eqref{eq:V_dot} that $\dot{V} = 0$ if and only if $k_1e_1 = e_2.$
\end{proof}

\begin{proposition}
    \label{prop:M_def}
    Consider the system \eqref{eq:z1_dot}, \eqref{eq:z2_dot}.
    Consider the function \eqref{eq:LyapCandidate}.
    Let $k_2 = k_1\inv.$
    Let $M \subset E$ such that
    $u = 0.$ 
    Then, $M$ is the largest positively invariant set in $E.$
\end{proposition}

\begin{proof}
    It follows from Proposition \ref{prop:equilibriumPoint} that the point $ z_1 = z_{1\rmd}, \, z_2 = u = 0 $ is an equilibrium of the system. 
    Hence, the set of all equilibrium points constitutes a positively invariant subset of $ E $. 
    Next, consider the case where $ u \neq 0 $ and $ z_2 = 0 $. 
    From \eqref{eq:z2_dot} and the fact that $ \SG_2(z_1, z_2) $ is nonsingular for all 
    $ z_1, z_2 \in \mathbb{R}^n $, it follows that $ \dot{z}_2 \neq 0 $. 
    Consequently, the trajectory leaves the set $ E $. 
    Therefore, $ M \subset E $ is the largest positively invariant subset of $ E $.
\end{proof}

\begin{theorem}
    \label{thrm:stability}
    Consider the system \eqref{eq:x1_dot}, \eqref{eq:x2_dot} and the system \eqref{eq:z1_dot}, \eqref{eq:z2_dot}.
    Assume that the initial conditions of \eqref{eq:x1_dot}, \eqref{eq:x2_dot} are inside the safe set, that is, 
    $(x_1(0),x_2(0)) \in \SSS,$ and let $|x_{1\rmd}|<\overline{x}_1.$
    Consider the control law \eqref{eq:u}, with $k_2 = k_1^{-1}.$
    Then,
    \begin{align}
        \lim_{t\to\infty}z_1(t) 
            &= 
                z_{1\rmd},
        \label{eq:z1limit}
        \\
        \lim_{t\to\infty}z_2(t) 
            &= 
                0,
        \\
        \lim_{t\to\infty}u(t) 
            &= 
                0,
        \label{eq:ulimit}
        \\
        \lim_{t\to\infty}x_1(t) 
            &= 
                x_{1\rmd},
        \label{eq:x1limit}
        \\
        \forall t\geq0,
        (x_1(t), x_2(t)) 
            &\in
                \SSS.
        \label{eq:safeset_FI}
    \end{align}
\end{theorem}
\begin{proof}
    Consider the set $\Omega$ defined in Proposition \ref{prop:Omega_def} and the set $E$ defined in Proposition \ref{prop:E_def}. 
    It follows from Proposition \ref{prop:M_def} that $M$ is the largest positively invariant set in $E.$
    It follows from Theorem 2 in \cite{lasalle1960some} that any solution starting in $\Omega$ converge to $M$ as $t\to\infty,$ which implies \eqref{eq:z1limit}--\eqref{eq:ulimit}.
    Next, it trivially follows from \eqref{eq:z1d} and \eqref{eq:z1limit} that $\lim_{t\to\infty}x_1(t) = x_{1\rmd}.$
    Finally, since the map $\psi$ is a bijection, it follows that $\forall t\geq0,
        (x_1(t), x_2(t)) \in \SSS.$
\end{proof}

\jp{
\begin{corollary}[Non-conservative invariance]
Under the conditions of Theorem \ref{thrm:stability}, the proposed control law guarantees forward invariance of the original safe set $\SSS$ directly. 
In particular, invariance is achieved without introducing auxiliary constraints on higher-order states, and the admissible invariant set coincides with the prescribed safe set $\SSS$.
\end{corollary}
}

\jp{\begin{remark}
The stability and forward invariance guarantees established above are nominal and assume exact knowledge of the system dynamics. 
The effects of model uncertainty and external disturbances are not considered in this work and remain important directions for future research.
\end{remark}}


\section{Motivation for the Sigmoid Integral Function as a Lyapunov Candidate}
\label{sec:justification_sigmoid_integral}

The sigmoid integral functions, introduced in Section \ref{sec:sig_integral}, are selected as the Lyapunov candidate functions for control synthesis because the classical quadratic Lyapunov function leads to a control law that exhibits numerical instability arising from specific ill-conditioned terms.
In contrast, employing the sigmoid integral function eliminates these problematic terms, yielding a numerically robust closed-loop system. The following example demonstrates the numerical difficulties associated with the classical quadratic Lyapunov function.
Consider the double integrator \eqref{eq:doubleIntegrator},
where $x_1,x_2,u \in \mathbb{R}.$ 
To simplify presentation, we consider the problem of only constraining $x_2$ with $\overline{x}_2 = 1.$ Note that $x_1$ is unconstrained. 
Following the procedure outlined in Section \ref{sec:state_transformation}, the transformed system in $z$ coordinates is
\begin{align}
    \dot{z}_1
        &=
            \tanh
            \left(
                z_2
            \right),
    \quad
    \dot{z}_2
        =
            \cosh^2
            \left(
                z_2
            \right)
            u.
    \label{eq:z2_dot_di}
\end{align}
Next, following the procedure outlined in Section \ref{sec:ControlSynthesis} with $\phi(x) = \atanh(x)$ as the constraint lifting function, the controller is
\begin{align}
    u
        &=
            -
            k_2
            \mathrm{sech}^2
            \left(
                z_2
            \right)
            \left(
                \tanh
                \left(
                    z_2
                \right)
                +
                k_1
                e_1
            \right),
    \label{eq:proposed_controller_di}
\end{align}
which implies that
\begin{align}
    \dot{z}_2
        &=
            -
            k_2
            \,
            \left(
                \tanh
                \left(
                    z_2
                \right)
                +
                k_1
                e_1
            \right).
    \label{eq:z2dot_proposed}
\end{align}

If, instead of the sigmoid integral function, the classical quadratic function were used, that is,
\begin{align}
    V 
        &= 
            V_1 
            + 
            \half e_2^\rmT e_2,
\end{align}
then
$
    \dot V 
        =
            - 
            k_1 e_1^2
           +
            e_2
            \left(
                e_1
                +
                k_1
                \tanh
                \left(
                    z_2
                \right)
                +
                u
            \right).
$
Furthermore, the control law
\begin{align}
    u 
        &=
            -
            \left(
                e_1
                +
                k_1
                \tanh
                \left(
                    z_2
                \right)
                +
                k_2
                e_2
            \right),
    \label{eq:class_controller_di}
\end{align}
where $k_1, k_2>0$ ensure that 
$\dot V = - k_1 e_1^2 - k_2 e_2^2 \leq 0.$
Moreover, $z_2$ satisfies
\begin{align}
    \dot z_2 
        &= 
            -
                \cosh^2
                \left(
                    z_2
                \right)
                \left(
                    e_1
                    +
                    k_1
                    \tanh
                    \left(
                        z_2
                    \right)
                    +
                    k_2
                    e_2
                \right).
    \label{eq:z2dot_classical}
\end{align}

Note that the variables $e_1$ and $e_2$ for both formulations are exactly the same and are given by $e_1 = z_1 - z_{1\rmd},$ and $e_2 = \tanh\left(z_2\right) + k_1e_1.$
A straightforward Lyapunov analysis shows that the closed-loop dynamics are asymptotically stable in both cases.
However, $\dot z_2$ in \eqref{eq:z2dot_classical} includes the term $\cosh(z_2)^2$, while $\dot z_2$ in \eqref{eq:z2dot_proposed} does not. 
As shown in Figure \ref{fig:problematic_term}, the function $\cosh(z_2)^2$ grows exponentially, implying that $\dot z_2$ in \eqref{eq:z2dot_classical} also increases exponentially, which leads to numerical instability. 

\begin{figure}[ht]
    \centering
    \includegraphics[width=\columnwidth]{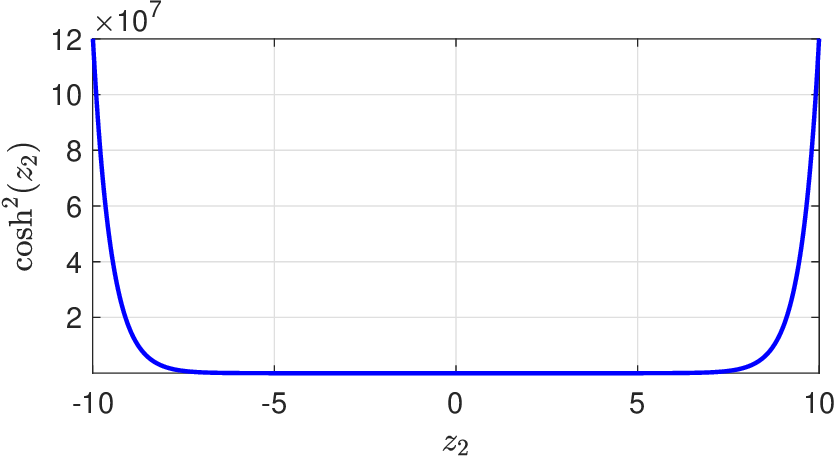}
    \caption{
    The function $\cosh^2(z_2).$ 
    Note the exponential growth in $\cosh^2(z_2).$ 
    For a value of $z_2 \approx 10,$ $\cosh^2(z_2) \approx  10^{7}.$
    }
    \label{fig:problematic_term}
\end{figure}

We emphasize that both controllers theoretically guarantee asymptotic stability, as established by the Lyapunov analysis.
However, the corresponding closed-loop dynamics exhibit numerical instability, which imposes practical limitations on implementation.
This behavior has been consistently observed in multiple numerical investigations with various choices of constraint lifting functions.
The numerical instability arising from the classical quadratic Lyapunov function thus motivates the use of sigmoid integral functions for constructing a numerically robust Lyapunov candidate.

\section{Application to Safe Attitude Control}
\label{sec:example}

This section considers the constrained attitude control problem.  
Although Euler angles offer an intuitive and compact representation of attitude dynamics and are often preferred over quaternions and direction cosine matrices for visualization purposes, they suffer from the well-known \textit{gimbal lock} phenomenon \cite{hemingway2018perspectives}.
Accordingly, the objective of this example is to design a control law that enables tracking of an arbitrary attitude command within the prescribed safe set, while ensuring the body's attitude remains strictly away from gimbal lock configuration. 

Consider a rigid body $\SB$ with a body fixed-frame $\rm F_D.$
Let $\rm F_A$ be an inertial frame. 
Assume that $\rm F_D$ is related to $\rm F_A$ by a 3-2-1 Euler angle sequence, that is, 
\begin{align}
    \rm F_A
        \rotation{\psi}{3}
    \rm F_B
        \rotation{\theta}{2}
    \rm F_C
        \rotation{\phi}{1}
    \rm F_D,
\end{align}
where
$\rm F_B$ is obtained by rotating $\rm F_A$ about its third axis, 
$\rm F_C$ is obtained by rotating $\rm F_B$ about its second axis, and 
$\rm F_D$ is obtained by rotating $\rm F_C$ about its first axis. 
Note that each rotation is an Euler rotation and the corresponding angle is the Euler angle. 

Let $q \isdef \matl \phi & \theta & \psi \matr^\rmT \in \BBR^3$ denote the vector of Euler angles $\phi, \theta,$ and $\psi.$
Then, the orientation of $\SB$ parameterized by $q$ satisfies the Poisson's equation
\begin{align}
    \dot q
        &=
            S(q)\inv \omega, 
    \label{eq:q_dot}    
\end{align}
where $\omega \in \BBR^3$ is the angular velocity vector of $\rm F_D$ relative to $\rm F_A$ measured in the body-fixed frame $\rm F_D$ and 
\begin{align}
S(q)
\isdef
    \matl 
        1&0&-\sin\theta \\
        0&\cos\phi &(\sin\phi)\cos\theta\\
        0&-\sin\phi &(\cos\phi)\cos\theta 
    \matr.
\label{eq:S_theta_for_321}
\end{align}
Note that, if $\theta \neq \pm \pi/2,$ then $S(q)$ is nonsingular. 
Otherwise, the inverse of $S(q),$ is given by 
\begin{align}
    S(q)\inv
        =
            \matl 
                1 & \sin(\phi) \tan(\theta) & \cos(\phi) \tan(\theta) \\
                0 & \cos(\phi) & -\sin(\phi) \\
                0 & \sin(\phi) \sec(\theta) & \cos(\phi) \sec(\theta) 
            \matr.
\end{align}
Note that the gimbal lock arises from the singularity of the matrix $S(q)$ at $\theta = \pm \pi/2$.
As $\theta$ approaches $\pm \pi/2,$ certain trigonometric terms in $S(q)$ become unbounded,  which can result in ill-conditioned computations and undesirable controller behavior.
The control objective is thus to enforce the constraint $\theta \in [-\overline \theta, \overline{ \theta}],$ where $\overline \theta< \pi/2.$

Next, the angular velocity $\omega$ satisfies the Euler's equation
\begin{align}
    J \dot \omega + \omega \times J \omega 
        =
            \tau,
    \label{eq:omega_dot}
\end{align}
where
$J \in \BBR^{3\times 3}$ is the moment of inertia matrix of $\SB$ in $\rm F_D$
and 
$\tau \in \BBR^3$ is the external torque applied to $\SB$ as measured in $\rm F_D.$

The attitude dynamics \eqref{eq:q_dot}, \eqref{eq:omega_dot} is in the form similar to \eqref{eq:x1_dot}, \eqref{eq:x2_dot}, where 
$x_1 = q,$ $x_2 = \omega,$
\begin{align}
    g_1(x_1) &= S(x_1)\inv,                     \\
    f_2(x_1,x_2) &= J \inv x_2 \times J x_2,     \\
    g_2(x_1,x_2) &= J\inv.
\end{align}
Note that Assumptions  \ref{ass:function_zero} and \ref{ass:function_invert} are satisfied.

In this numerical example, the safe set is defined by  
$\overline{q} \isdef 
            \matl
                \pi/3
                &
                \pi/3
                &
                \pi/3
            \matr^\rmT,$
and 
$\overline{\omega}
        \isdef
            \matl
                0.03
                &
                0.02
                &
                0.01
            \matr^\rmT.
$
These constraints imply that each of the Euler angles should remain smaller than $\pi/3$ radians in magnitude and the components of the angular velocity vector should not exceed $0.03$ $\rm rad/s,$ $0.02$ $\rm rad/s,$ and $0.01$ $\rm rad/s,$ respectively in magnitude.

To simulate the attitude dynamics, the body-fixed frame $\rm F_D$ is assumed to coincide with the principal-axis frame.
Consequently, the inertia matrix is diagonal, that is, $J = {\rm diag}(J_1, J_2, J_3),$ where we set $J_1 = 1,$ $J_2 = 2,$ and $J_3 = 3$ in this example. 
The initial conditions to simulate \eqref{eq:q_dot} and \eqref{eq:omega_dot} and the attitude commands in $\SSS$ are randomly generated from a uniform distribution over $\SSS$.
In particular, we randomly generate ten initial conditions and attitude commands. 
%

The attitude controller is designed following the procedure outlined in Section \ref{sec:controller}, with the control law given by \eqref{eq:u}. 
In particular, we use the $\phi(x) = \mathrm{atanh}(x)$ as the constraint-lifting function. 
Consequently, $\psi(z)$ and $\SV(\zeta)$ are given by the second line in Table \ref{tab:sat_functions}.
In this numerical example, the control gain is set to $k_1 = 0.1$.

\jp{Figure \ref{fig:attitude_position_random} presents a geometric illustration of the closed-loop Euler angle trajectories under multiple initial conditions and reference commands. %
The pink-shaded region represents the prescribed safe set, and all trajectories remain confined within this set, demonstrating forward invariance. %
For each trajectory, the initial attitude is marked by a triangle, and the corresponding desired value is indicated by a square of the same color.}
\begin{figure}[h]
    \centering
    \includegraphics[width=\columnwidth, trim={7.0cm 0 6.5cm 0}, clip]{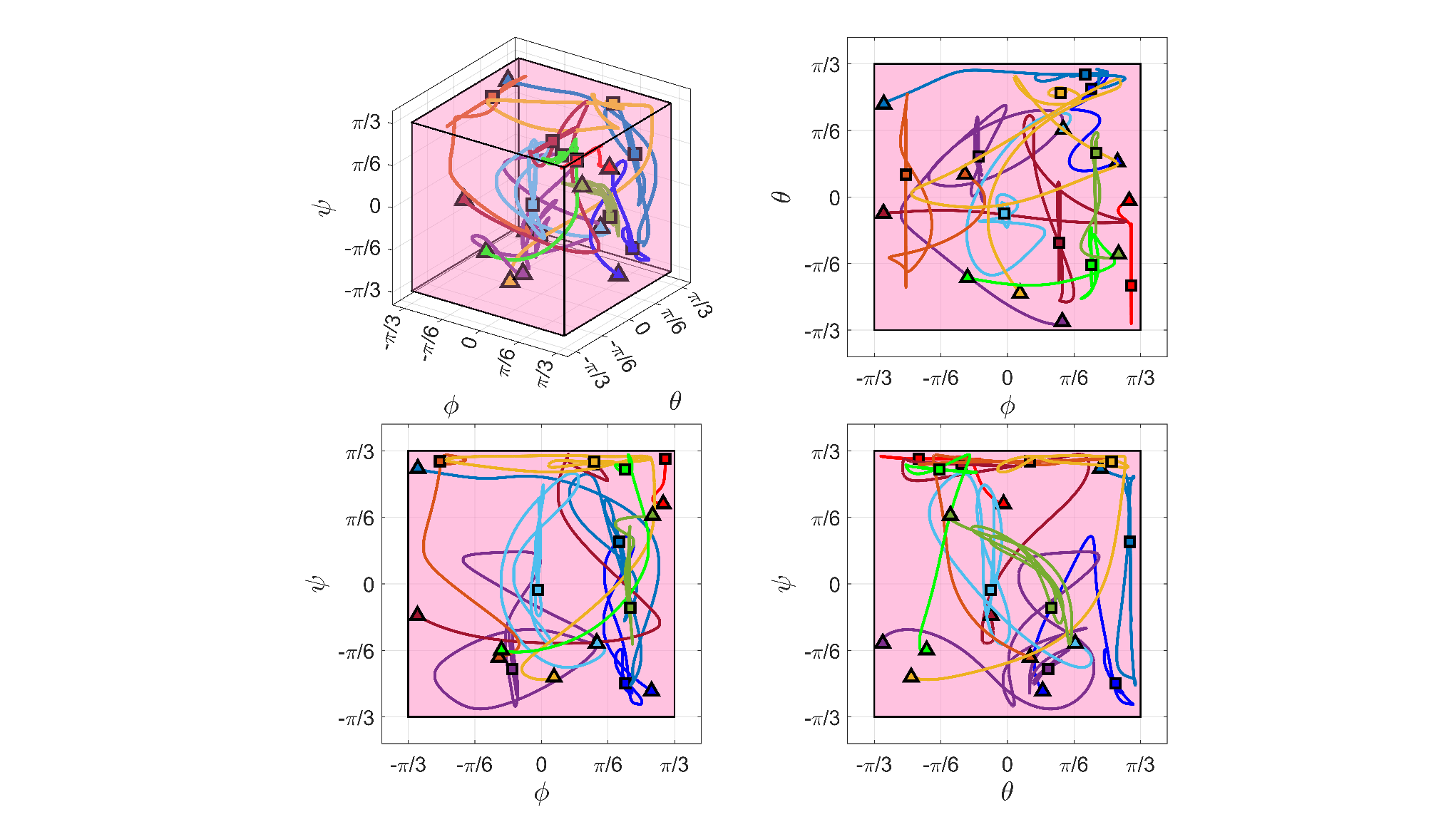}
    \caption{
    \jp{Geometric illustration of forward invariance for the closed-loop system under multiple initial conditions and reference commands. 
    The trajectories remain confined within the prescribed safe set $\SSS$ (shown in pink). 
    Desired angle values are marked with squares, and initial conditions are marked with triangles.}
    }
    \label{fig:attitude_position_random}
\end{figure}

\jp{Figure \ref{fig:attitude_velocity_random} presents a geometric illustration of the closed-loop angular velocity trajectories under multiple initial conditions. %
The yellow-shaded region represents the prescribed safe set, and all trajectories remain confined within this set, demonstrating forward invariance. %
The initial angular velocities are marked by triangles, and the desired equilibrium (zero angular velocity) is marked at the origin. %
As expected, all trajectories converge to zero.}
\begin{figure}[h]
    \centering
    \includegraphics[width=\columnwidth, trim={0.9cm 0 1.5cm 0}, clip]{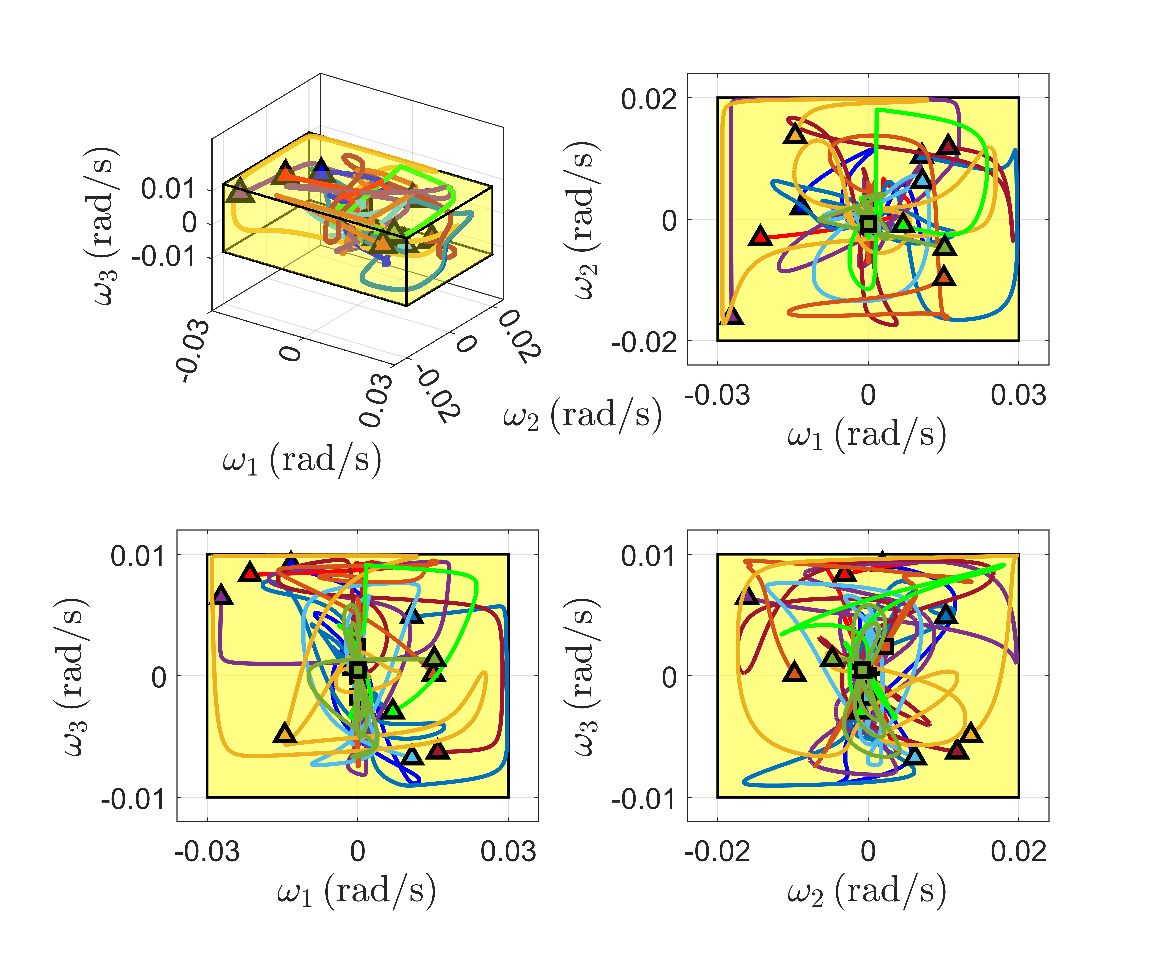}
    \caption{
    \jp{
    Geometric illustration of forward invariance for the closed-loop angular velocity trajectories under multiple initial conditions. %
    The trajectories remain confined within the prescribed safe set $\SSS$ (shown in yellow). %
    Initial conditions are marked with triangles.
    }
    }
    \label{fig:attitude_velocity_random}
\end{figure}

Finally, Figure \ref{fig:attitude states_control} shows the time responses of the closed-loop Euler angles, angular velocities, and control torques for the ten trajectories.
The first column shows the Euler angle responses, the second column shows the angular velocity responses, and the third column shows the control torque signals.
The pink- and yellow-shaded regions in the first and second columns, respectively, denote the corresponding safe sets.
\begin{figure}[h]
    \centering
    \includegraphics[width=\columnwidth, trim={2.0cm 2.5 2.0cm 0}, clip]{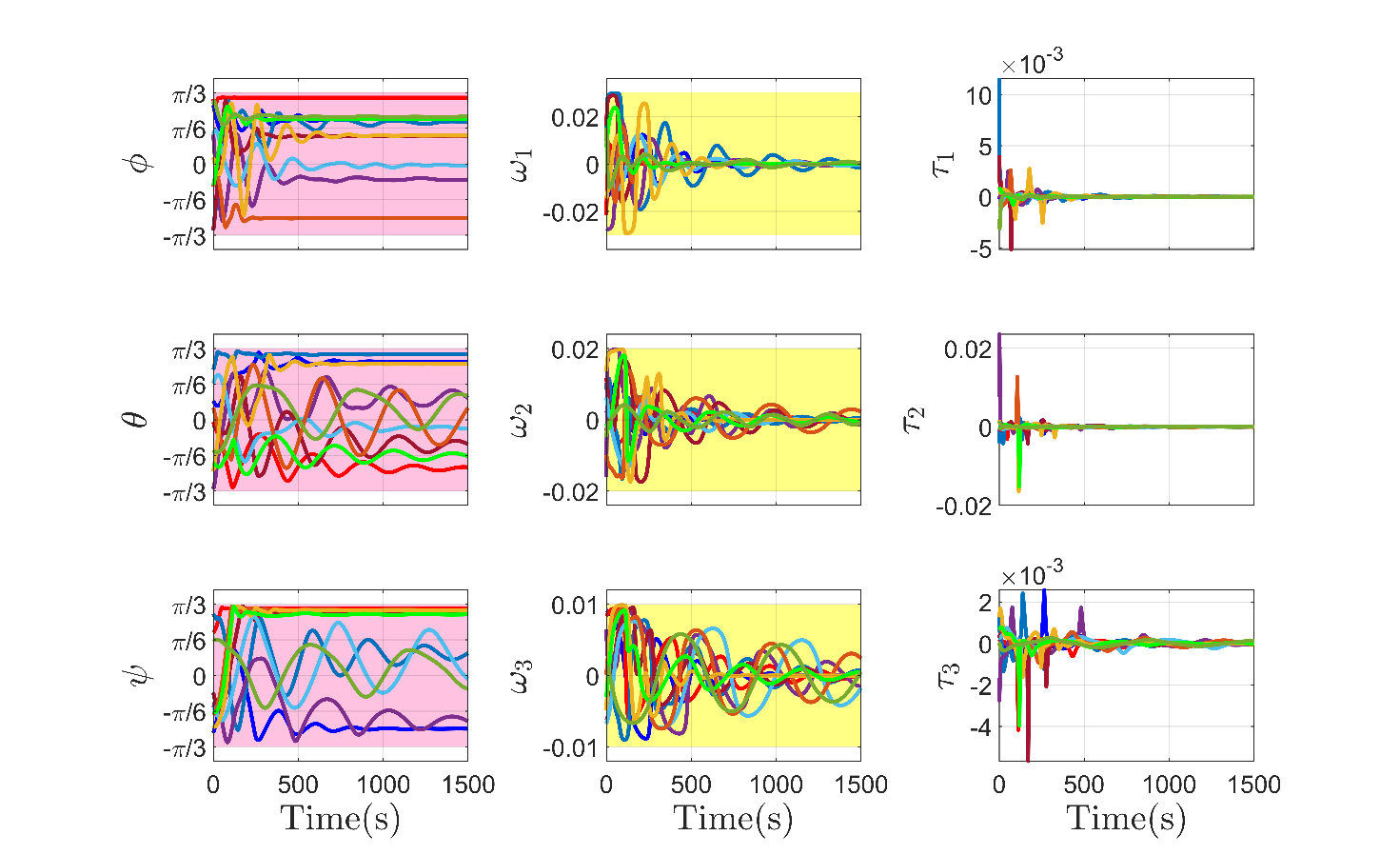}
    \caption{
    Closed-loop responses of the attitude dynamics: angles $(\phi,\theta,\psi)$ and angular velocities $(\omega_1,\omega_2,\omega_3)$ over time for randomly initialized and commanded trajectories, and control inputs $(\tau_1,\tau_2,\tau_3)$ generated by the control law \eqref{eq:u}. Note that the angle and angular responses are always contained within the desired safe set $\SSS.$
    }
    \label{fig:attitude states_control}
\end{figure}

\section{Conclusions}
\label{sec:conclusions}

This paper presented a constraint-enforcing control framework that guarantees closed-loop stability for a class of nonlinear systems while ensuring the positive invariance of a prescribed safe set.
\jp{
A key outcome of this work is that forward invariance of the original safe set can be enforced for constraints with relative degree greater than one, without introducing auxiliary conditions on higher-order states.
}
\jp{
In contrast to higher-order barrier constructions, which typically enforce invariance on a modified set defined by the intersection of multiple constraints, the proposed approach directly preserves the original constraint set, thereby avoiding the conservatism associated with derivative-based conditions.
}
The proposed approach transforms the constrained control problem into an equivalent unconstrained one through \textit{constraint-lifting functions} constructed using strictly increasing sigmoid functions.
A class of Lyapunov candidate functions, called \textit{sigmoid integral functions}, was introduced. 
These functions behave quadratically near the origin and linearly far away from it, a key property that eliminates exponentially growing terms that lead to numerical instability in the classical Lyapunov-based controller design.
The sigmoid integral functions are then employed in the controller synthesis to design a stabilizing feedback law for the transformed system.
A rigorous theoretical analysis establishes the asymptotic stability of the resulting closed-loop system.
\jp{
The development in this paper focuses on a two-level system structure for clarity of presentation. Nevertheless, the proposed framework extends to systems with higher relative degree, in which the control synthesis follows a recursive structure analogous to backstepping, thereby increasing design complexity.
}
A numerical simulation demonstrates the effectiveness of the proposed technique for a constrained attitude control problem.

\appendices
\section{Dynamics in \texorpdfstring{$z$}{z} Coordinate}
\label{ap:derivation of z_dynamics}
It follows from \eqref{eq:x1_to_z1} that 
\begin{align}
    \dot{z}_1
        &=
            D(\overline{x}_1)
            \partial_{\chi_1} 
            \phi(\chi_1) 
            \dot \chi_1
        \nn \\ 
        &=
            D(\overline{x}_1)
            \partial_{\chi_1} 
            \phi(\chi_1) 
            D_{\rmI}(\overline{x}_1)
            \dot x_1
        \nn \\ 
        &=
            \partial_{\chi_1} 
            \phi(\chi_1)
            g_1(x_1)
            x_2
        \nn \\ 
        &=
            \partial_{\chi_1} 
            \phi(\psi_1(\zeta_1)) 
            g_1(D(\overline{x}_1)\psi_1(\zeta_1))
            D(\overline{x}_2)
            \psi_2(\zeta_2)
        \nn \\
        &=
            \Phi(\zeta_1) 
            \psi_2(\zeta_2)
        \nn \\
        &=
            \SG_1(z_1,z_2).
\end{align}
Similarly, it follows from \eqref{eq:x2_to_z2} that 
\begin{align}    
    \dot{z}_2
        &=
            D(\overline{x}_2)
            \partial_{\chi_2} 
            \phi(\chi_2) 
            \dot \chi_2
        \nn \\ 
        &=
            \partial_{\chi_2} 
            \phi(\chi_2)
            D(\overline{x}_2)
            D_{\rmI}(\overline{x}_2)
            \dot x_2
        \nn \\ 
        &=
            \partial_{\chi_2} 
            \phi(\chi_2)
            \left(
                f_2(x_1, x_2) 
                + 
                g_2(x_1,x_2) u
            \right)
        \nn \\ 
        &=
            \partial_{\chi_2} 
            \phi(\psi_2(\zeta_2))
            \Big(
                f_2
                (
                    D(\overline{x}_1)
                    \psi_1(\zeta_1),
                    D(\overline{x}_2)
                    \psi_2(\zeta_2)
                ) 
        \nn\\
        & \quad
                + 
                g_2
                (
                    \jp{D(\overline{x}_1)}\psi_1(\zeta_1),
                    D(\overline{x}_2)\psi_2(\zeta_2)
                ) 
                u 
            \Big)
        \nn \\
        &=
            \SF_2(z_1,z_2)  + \SG_2(z_1,z_2) u.
\end{align}

\section{Time-derivative of Sigmoid Integral-based Lyapunov Function}
\label{ap:appendix_sigmoid integral definition}
Note that 
\begin{align}
    \rmd_t 
    \sum_{i=1}^n
            \SV(\zeta_{2i})
                &=
                    \sum_{i=1}^n
                   \rmd_t
                   \int_0^{\zeta_{2i}}
                   \sigma(s)
                   \rmd s
\end{align}
For each $i = 1, \ldots, n,$ it follows from the Leibniz's formula that
\begin{align}
    \rmd_t
       \int_0^{\zeta_{2i}}
       \sigma(s)
       \rmd s
        =
            \sigma(\zeta_{2i})
            \dot{\zeta}_{2i},
\end{align}
which, using \eqref{eq:zeta2_def}, implies that
\begin{align}
    \rmd_t 
    \sum_{i=1}^n
            \SV(\zeta_{2i})
                &=
                    \sum_{i=1}^n
                   \sigma(\zeta_{2i})
                   \dot{\zeta}_{2i}
                =
                    \psi(\zeta)^\rmT
                   \dot{\zeta}_{2}
                \nn \\
                &=
                    \psi(\zeta_2)^\rmT
                    D_{\rmI}(\overline{x}_2)\dot{z}_2. 
\end{align}

\bibliographystyle{IEEEtran}
\bibliography{Bib/b}

\begin{thebibliography}{10}
\providecommand{\url}[1]{#1}
\csname url@samestyle\endcsname
\providecommand{\newblock}{\relax}
\providecommand{\bibinfo}[2]{#2}
\providecommand{\BIBentrySTDinterwordspacing}{\spaceskip=0pt\relax}
\providecommand{\BIBentryALTinterwordstretchfactor}{4}
\providecommand{\BIBentryALTinterwordspacing}{\spaceskip=\fontdimen2\font plus
\BIBentryALTinterwordstretchfactor\fontdimen3\font minus \fontdimen4\font\relax}
\providecommand{\BIBforeignlanguage}[2]{{%
\expandafter\ifx\csname l@#1\endcsname\relax
\typeout{** WARNING: IEEEtran.bst: No hyphenation pattern has been}%
\typeout{** loaded for the language `#1'. Using the pattern for}%
\typeout{** the default language instead.}%
\else
\language=\csname l@#1\endcsname
\fi
#2}}
\providecommand{\BIBdecl}{\relax}
\BIBdecl

\bibitem{graichen2010handling}
K.~Graichen, A.~Kugi, N.~Petit, and F.~Chaplais, ``Handling constraints in optimal control with saturation functions and system extension,'' \emph{Systems \& Control Letters}, vol.~59, no.~11, pp. 671--679, 2010.

\bibitem{garone2017reference}
E.~Garone, S.~Di~Cairano, and I.~Kolmanovsky, ``Reference and command governors for systems with constraints: A survey on theory and applications,'' \emph{Automatica}, vol.~75, pp. 306--328, 2017.

\bibitem{wang2023balanced}
P.~Wang, H.~Wang, X.~Zhang, and S.~S. Ge, ``Balanced control between performance and saturation for constrained nonlinear systems,'' \emph{International Journal of Robust and Nonlinear Control}, vol.~33, no.~10, pp. 5675--5690, 2023.

\bibitem{gilbert1995discrete}
E.~G. Gilbert, I.~Kolmanovsky, and K.~T. Tan, ``Discrete-time reference governors and the nonlinear control of systems with state and control constraints,'' \emph{International Journal of robust and nonlinear control}, vol.~5, no.~5, pp. 487--504, 1995.

\bibitem{gilbert1999fast}
E.~G. Gilbert and I.~Kolmanovsky, ``Fast reference governors for systems with state and control constraints and disturbance inputs,'' \emph{International Journal of Robust and Nonlinear Control: IFAC-Affiliated Journal}, vol.~9, no.~15, pp. 1117--1141, 1999.

\bibitem{kim2019backstepping}
Y.~Kim, T.~H. Oh, T.~Park, and J.~M. Lee, ``Backstepping control integrated with lyapunov-based model predictive control,'' \emph{Journal of Process Control}, vol.~73, pp. 137--146, 2019.

\bibitem{kim2018integration}
Y.~Kim, T.~Park, and J.~M. Lee, ``Integration of model predictive control and backstepping approach and its stability analysis,'' \emph{IFAC-PapersOnLine}, vol.~51, no.~18, pp. 405--410, 2018.

\bibitem{SCHLAGENHAUF20207026}
\BIBentryALTinterwordspacing
J.~Schlagenhauf, P.~Hofmeier, T.~Bronnenmeyer, R.~Paelinck, and M.~Diehl, ``Cascaded nonlinear mpc for realtime quadrotor position tracking⁎⁎this research has been funded by the company kiteswarms in the kite project (zvk2017022302) with the university of freiburg.'' \emph{IFAC-PapersOnLine}, vol.~53, no.~2, pp. 7026--7032, 2020, 21st IFAC World Congress. [Online]. Available: \url{https://www.sciencedirect.com/science/article/pii/S2405896320307357}
\BIBentrySTDinterwordspacing

\bibitem{wang2020}
D.~Wang, C.~Zhao, J.~Hu, and Q.~Pan, ``Model predictive path following control of a quadrotor in constrained environments,'' in \emph{2020 IEEE 16th International Conference on Control and Automation (ICCA)}, 2020, pp. 719--724.

\bibitem{li2023adaptive}
J.~Li, L.~Wan, J.~Li, and K.~Hou, ``Adaptive backstepping control of quadrotor uavs with output constraints and input saturation,'' \emph{Applied Sciences}, vol.~13, no.~15, p. 8710, 2023.

\bibitem{tee2009barrier}
K.~P. Tee, S.~S. Ge, and E.~H. Tay, ``Barrier lyapunov functions for the control of output-constrained nonlinear systems,'' \emph{Automatica}, vol.~45, no.~4, pp. 918--927, 2009.

\bibitem{soukkou2023tuning}
Y.~Soukkou, M.~Tadjine, A.~Soukkou, M.~Nibouche, and H.~Nouri, ``Tuning functions based adaptive backstepping control for uncertain strict-feedback nonlinear systems using barrier lyapunov functions with full state constraints,'' \emph{European Journal of Control}, vol.~70, p. 100783, 2023.

\bibitem{ngo2005integrator}
K.~B. Ngo, R.~Mahony, and Z.-P. Jiang, ``Integrator backstepping using barrier functions for systems with multiple state constraints,'' in \emph{Proceedings of the 44th IEEE Conference on Decision and Control}.\hskip 1em plus 0.5em minus 0.4em\relax IEEE, 2005, pp. 8306--8312.

\bibitem{niu2024adaptive}
B.~Niu, X.~Wang, X.~Wang, X.~Wang, and T.~Li, ``Adaptive barrier-lyapunov-functions based control scheme of nonlinear pure-feedback systems with full state constraints and asymptotic tracking performance,'' \emph{Journal of Systems Science and Complexity}, vol.~37, no.~3, pp. 965--984, 2024.

\bibitem{Sadeghzadeh2023}
N.~Sadeghzadeh-Nokhodberiz and N.~Meskin, ``Consensus-based distributed formation control of multi-quadcopter systems: Barrier lyapunov function approach,'' \emph{IEEE Access}, vol.~11, pp. 142\,916--142\,930, 2023.

\bibitem{yin2019barrier}
Z.~Yin, B.~Wang, C.~Du, and Y.~Zhang, ``Barrier-lyapunov-function-based backstepping control for pmsm servo system with full state constraints,'' in \emph{2019 22nd International Conference on Electrical Machines and Systems (ICEMS)}.\hskip 1em plus 0.5em minus 0.4em\relax IEEE, 2019, pp. 1--5.

\bibitem{li2020barrier}
Y.-X. Li, ``Barrier lyapunov function-based adaptive asymptotic tracking of nonlinear systems with unknown virtual control coefficients,'' \emph{Automatica}, vol. 121, p. 109181, 2020.

\bibitem{tee2011control}
K.~P. Tee and S.~S. Ge, ``Control of nonlinear systems with partial state constraints using a barrier lyapunov function,'' \emph{International Journal of Control}, vol.~84, no.~12, pp. 2008--2023, 2011.

\bibitem{di2014stabilizing}
S.~Di~Cairano, W.~M.~H. Heemels, M.~Lazar, and A.~Bemporad, ``Stabilizing dynamic controllers for hybrid systems: A hybrid control lyapunov function approach,'' \emph{IEEE Transactions on Automatic Control}, vol.~59, no.~10, pp. 2629--2643, 2014.

\bibitem{zhu2024adaptive}
B.~Zhu, N.~Xu, G.~Zong, and X.~Zhao, ``Adaptive optimized backstepping tracking control for full-state constrained nonlinear strict-feedback systems without using barrier lyapunov function method,'' \emph{Optimal Control Applications and Methods}, vol.~45, no.~5, pp. 2051--2075, 2024.

\bibitem{taylor2022safe}
A.~J. Taylor, P.~Ong, T.~G. Molnar, and A.~D. Ames, ``Safe backstepping with control barrier functions,'' in \emph{2022 IEEE 61st Conference on Decision and Control (CDC)}.\hskip 1em plus 0.5em minus 0.4em\relax IEEE, 2022, pp. 5775--5782.

\bibitem{Kim2023}
J.~Kim and Y.~Kim, ``Safe control synthesis for multicopter via control barrier function backstepping,'' in \emph{2023 62nd IEEE Conference on Decision and Control (CDC)}, 2023, pp. 8720--8725.

\bibitem{Kim2025}
K.~H. Kim, M.~Diagne, and M.~Krstić, ``Constant-sum high-order barrier functions for safety between parallel boundaries,'' \emph{IEEE Control Systems Letters}, vol.~9, pp. 1447--1452, 2025.

\bibitem{romdlony2016stabilization}
M.~Z. Romdlony and B.~Jayawardhana, ``Stabilization with guaranteed safety using control lyapunov--barrier function,'' \emph{Automatica}, vol.~66, pp. 39--47, 2016.

\bibitem{ames2019control}
A.~D. Ames, S.~Coogan, M.~Egerstedt, G.~Notomista, K.~Sreenath, and P.~Tabuada, ``Control barrier functions: Theory and applications,'' in \emph{2019 18th European control conference (ECC)}.\hskip 1em plus 0.5em minus 0.4em\relax Ieee, 2019, pp. 3420--3431.

\bibitem{wu2019control}
Z.~Wu, F.~Albalawi, Z.~Zhang, J.~Zhang, H.~Durand, and P.~D. Christofides, ``Control lyapunov-barrier function-based model predictive control of nonlinear systems,'' \emph{Automatica}, vol. 109, p. 108508, 2019.

\bibitem{nagumo1942lage}
M.~Nagumo, ``\BIBforeignlanguage{German}{Über die lage der integralkurven gewöhnlicher differentialgleichungen},'' \emph{\BIBforeignlanguage{German}{Nippon Sugaku-Butsurigakkwai Kizi Dai 3 Ki}}, 1942.

\bibitem{blanchini1999set}
F.~Blanchini, ``Set invariance in control,'' \emph{Automatica}, vol.~35, no.~11, pp. 1747--1767, 1999.

\bibitem{ames2016control}
A.~D. Ames, X.~Xu, J.~W. Grizzle, and P.~Tabuada, ``Control barrier function based quadratic programs for safety critical systems,'' \emph{IEEE Transactions on Automatic Control}, vol.~62, no.~8, pp. 3861--3876, 2016.

\bibitem{ong2024rectified}
P.~Ong, M.~H. Cohen, T.~G. Molnar, and A.~D. Ames, ``Rectified control barrier functions for high-order safety constraints,'' \emph{IEEE Control Systems Letters}, 2024.

\bibitem{guo2014backstepping}
T.~Guo and X.~Wu, ``Backstepping control for output-constrained nonlinear systems based on nonlinear mapping,'' \emph{Neural Computing and Applications}, vol.~25, pp. 1665--1674, 2014.

\bibitem{huang2020output}
X.~Huang, Y.~Song, and C.~Wen, ``Output feedback control for constrained pure-feedback systems: A non-recursive and transformational observer based approach,'' \emph{Automatica}, vol. 113, p. 108789, 2020.

\bibitem{sampei2003nonlinear}
M.~Sampei, J.~Shen, and N.~H. McClamroch, ``Nonlinear control of the air spindle testbed with constraints,'' in \emph{Proceedings of the 2003 American Control Conference, 2003.}, vol.~1.\hskip 1em plus 0.5em minus 0.4em\relax IEEE, 2003, pp. 483--488.

\bibitem{menon1996characterization}
A.~Menon, K.~Mehrotra, C.~K. Mohan, and S.~Ranka, ``Characterization of a class of sigmoid functions with applications to neural networks,'' \emph{Neural networks}, vol.~9, no.~5, pp. 819--835, 1996.

\bibitem{lasalle1960some}
J.~LaSalle, ``Some extensions of liapunov's second method,'' \emph{IRE Transactions on circuit theory}, vol.~7, no.~4, pp. 520--527, 1960.

\bibitem{hemingway2018perspectives}
E.~G. Hemingway and O.~M. O’Reilly, ``Perspectives on euler angle singularities, gimbal lock, and the orthogonality of applied forces and applied moments,'' \emph{Multibody system dynamics}, vol.~44, no.~1, pp. 31--56, 2018.

\end{thebibliography}

\begin{IEEEbiography}[{\includegraphics[width=1in,height=1.25in,clip,keepaspectratio]{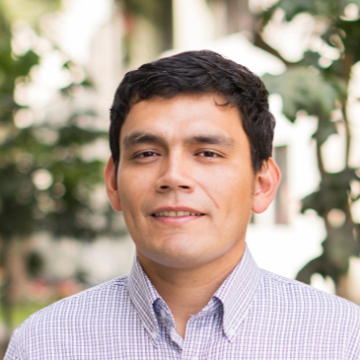}}]{Jhon Portella} received the B.Sc. degree in Mechanical Engineering from the Pontificia Universidad Católica del Perú (PUCP), Lima, Peru, in 2013, and the M.Eng. degree in Mechatronics Engineering from the University of Melbourne, Melbourne, Australia, in 2016. He is currently pursuing a Ph.D. degree in Mechanical Engineering at the University of Maryland, Baltimore County, Baltimore, MD, USA. His research interests include nonlinear and adaptive control, as well as the use of optimization and machine learning techniques for control design.
\end{IEEEbiography}
\begin{IEEEbiography}[{\includegraphics[width=1in,height=1.25in,clip,keepaspectratio]{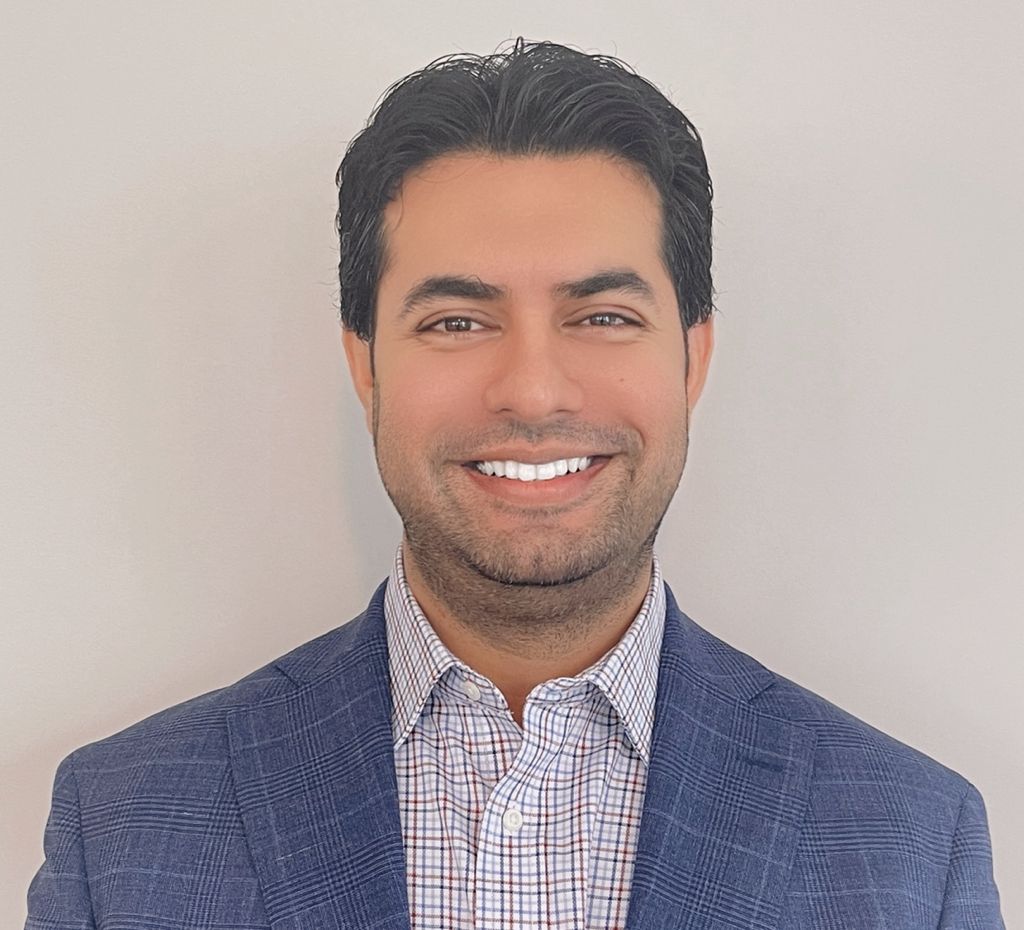}}]{Ankit Goel} (Member, IEEE) received the B.E. degree in Mechanical Engineering from Delhi College of Engineering, Delhi, in 2009, and the M.S. and Ph.D. degrees in Aerospace Engineering from The University of Michigan, Ann Arbor, MI, USA, in 2014 and 2019, respectively. 
He is currently an Assistant Professor in the Department of Mechanical Engineering at the University of Maryland, Baltimore County, where he directs the Estimation, Control, and Learning Laboratory (ECLL). His research interests include model-free, data-driven, and learning-based adaptive control, as well as model-based methods for complex dynamical systems.
\end{IEEEbiography}

\end{document}